\documentclass[11pt,a4paper]{amsart}

\usepackage{amsthm}
\usepackage{amsfonts}
\usepackage{amsmath}
\usepackage{amssymb}
\usepackage{wrapfig}
\usepackage{accents}
\usepackage{graphicx}
\graphicspath{{./figure/MCG/}{./figure/homology_of_C/}{./figure/Dehn_twist/}{./figure/braid_twist/}{./figure/homology_of_surface/}{./figure/generators_of_Hm/}{./figure/generators_of_H1Sigma_g,1/}{./figure/Ygn/}{./figure/generators_of_MCG/}{./figure/intersection/}{./figure/self_map/}{./figure/forketc/}{./figure/k-prong/}}

\theoremstyle{definition}
\newtheorem{Th}{Theorem}[section]

\theoremstyle{definition}

\theoremstyle{definition}
\newtheorem{Rem}[Th]{Remark}

\theoremstyle{definition}
\newtheorem{Prop}[Th]{Proposition}

\theoremstyle{definition}
\newtheorem{Cor}[Th]{Corollary}

\theoremstyle{definition}
\newtheorem{Lem}[Th]{Lemma}

\theoremstyle{definition}

\theoremstyle{definition}

\theoremstyle{definition}

\theoremstyle{definition}

\makeatletter
  
  \@addtoreset{equation}{section}
\makeatother

\newcommand{\relmiddle}[1]{\mathrel{}#1\mathrel{}}

\newcommand{\Rplus}{\mathbb{R}_+} 
\newcommand{\torus}[1]{T_{#1}} 
\newcommand{\Fix}{\operatorname{Fix}} 
\newcommand{\eFix}{\operatorname{eFix}} 
\newcommand{\FPC}[1]{{\bf #1}} 
\newcommand{\cd}[1]{\mathrm{cd}(#1)} 
\newcommand{\cdGamma}[1]{\mathrm{cd}_\Gamma(#1)} 
\newcommand{\ind}[2]{\mathrm{ind}(#1,#2)} 
\newcommand{\freeloop}[1]{{[[#1]]}} 

\newcommand{\norm}[1]{\left\|#1\right\|} 
\newcommand{\Growth}[1]{\operatorname{\underset{#1\to\infty}{\operatorname{Growth}}}} 

\newcommand{\mc}{\beta} 
\newcommand{\representative}{f_\mc} 

\newcommand{\X}{D^2} 
\newcommand{\Punc}{P_n} 
\newcommand{\Sn}{D_n} 

\newcommand{\C}{\mathcal{C}} 
\newcommand{\PureConf}{F_{n,m}(\X)} 
\newcommand{\Conf}{\C_{n,m}(\X)} 
\newcommand{\IntertwiningConf}{IT_{n,m}(\X)} 

\newcommand{\Braid}{{\bf B}_{n,m}(\X)} 
\newcommand{\IntertwiningBraidn}[1]{{\bf E}_{n,#1}(\X)} 
\newcommand{\IntertwiningBraid}{\IntertwiningBraidn{m}} 
\newcommand{\AlgBraid}{B_{n,m}} 
\newcommand{\AlgIntertwiningBraidn}[1]{E_{n,#1}} 
\newcommand{\AlgIntertwiningBraid}{\AlgIntertwiningBraidn{m}} 
\newcommand{\BraidR}{R_B} 
\newcommand{\IntertwiningR}{R} 

\newcommand{\ConfCov}{p} 

\newcommand{\Confmap}[1]{\widehat{#1}} 

\newcommand{\Y}{Y_n} 
\newcommand{\cmc}{\overline{\representative}} 
\newcommand{\interior}{\operatorname{int}} 

\newcommand{\lift}[1]{\widetilde{#1}} 
\newcommand{\Conflift}[1]{\lift{#1}} 

\newcommand{\Hm}[1]{H_m(#1)} 

\newcommand{\IP}[2]{\left\langle#1,#2\right\rangle} 
\newcommand{\homology}[1]{\left[#1\right]} 

\newcommand{\Euler}{n-1} 
\newcommand{\torusfund}{\Gamma_{\mc,m}} 
\newcommand{\trace}{\tr_{\torusfund}\rep(\mc)} 
\newcommand{\tracek}[1]{\tr_{\Gamma_{\mc^{#1},m}}\rep(\mc^{#1})} 
\newcommand{\emb}{g} 
\newcommand{\dil}{\lambda} 
\newcommand{\rhomatrix}{M(n,\mu,\nu)} 

\newcommand{\rep}{\zeta_{n,m}} 
\newcommand{\basis}{\mathcal{E}_{n,m}} 

\newcommand{\Confunivcov}{\Conflift{\C}_{n,m}(\X)} 

\newcommand{\LKB}{\rho_{\mathit{LKB}}} 
\newcommand{\Burau}{\rho_B} 

\newcommand{\Aut}{\operatorname{Aut}} 
\newcommand{\GL}{\operatorname{GL}} 
\newcommand{\id}{\mathrm{id}} 
\newcommand{\Ker}{\operatorname{Ker}} 
\newcommand{\sgn}{\operatorname{sgn}} 
\newcommand{\tr}{\operatorname{tr}} 
\newcommand{\st}{\mathrm{\:s.t.\:}} 
\newcommand{\Gauss}[1]{\lfloor#1\rfloor} 
\newcommand{\sr}{\operatorname{sr}} 

\title[Fixed point theory and dilatation]{A new relationship between the dilatation of pseudo-Anosov braids and fixed point theory}
\author{Yumehito Kawashima}

\begin{document}

\begin{abstract}



A relation between the dilatation of pseudo-Anosov braids and fixed point theory was studied by Ivanov. In this paper we reveal a new relationship between
the above two subjects by showing a formula 
for the dilatation of pseudo-Anosov braids by means of the representations of braid groups due to B. Jiang and H. Zheng. 
\end{abstract}

\maketitle

\section{Introduction}\label{sec:introduction}






The purpose of this paper is to reveal a new relationship between the dilatation of pseudo-Anosov braids and fixed point theory. For this purpose we obtain a new formula to determine the dilatation of pseudo-Anosov braids from the representation $\rep$ due to Jiang and Zheng \cite{forcing_relation}. 

Let us recall the notion of pseudo-Anosov braids. Let $\Sigma_g$ be a closed surface of genus $g$ and $\Punc$ be an $n$-point subset of $\Sigma_g$. We denote by $\Sigma_{g,n}$ the subset of $\Sigma_g$ deleting $\Punc$. We consider the case when $\Sigma_{g,n}$ has negative Euler characteristic. 
Let $f$ be a homeomorphism of $\Sigma_g$ fixing $\Punc$ setwise. We recall that $f$ is \emph{periodic} if $f^k$ equals identity for some $k>0$, and it is \emph{reducible} if there exists an $f$-invariant closed $1$-manifold $J\subset\Sigma_{g,n}$ whose complementary components in $\Sigma_{g,n}$ have negative Euler characteristic or else are M\"obius bands. We refer to $J$ as a \emph{reduction} of $f$. Finally, $f$ is \emph{pseudo-Anosov} if there exists a number $\lambda>1$ and a pair $\mathcal{F}^s$, $\mathcal{F}^u$ of transverse measured foliations with singularities modelled on $k$-prongs, $k = 1,2,\dots$ in Figure \ref{fig:k-prong} such that the equalities $f(\mathcal{F}^s) = (1\slash\lambda)\mathcal{F}^s$ and $f(\mathcal{F}^u)=\lambda\mathcal{F}^u$ hold. Furthermore, the one-prong singularities of these foliations are allowed to occur only at the punctures. 
For an isotopy class $\varphi$ of homeomorphisms of $\Sigma_g$, $\varphi$ is \emph{periodic} if there exists a periodic element in $\varphi$. Similarly, $\varphi$ is \emph{reducible} if there exists a reducible element in $\varphi$ and $\varphi$ is \emph{pseudo-Anosov} if there exists a pseudo-Anosov element in $\varphi$. 

\begin{figure}[tbp]
\begin{center}
\includegraphics[width=\textwidth]{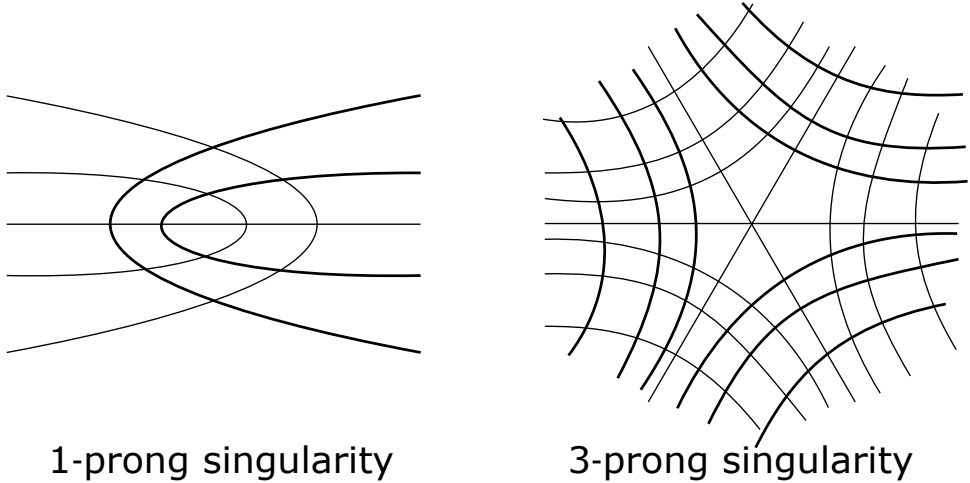}
\end{center}
\caption{local chart around the singularities}
\label{fig:k-prong}
\end{figure}

In \cite{NT_classification}, Thurston classified the isotopy classes of homeomorphisms on $\Sigma_g$ fixing $P_n$ into periodic, reducible and pseudo-Anosov. Since we can regard the braid group $B_n$ on $n$ strands as the mapping class group of disk with $n$ punctures, every element of $B_n$ is also classified into periodic, reducible and pseudo-Anosov types. In \cite{train_track}, Bestvina and Handel obtained an algorithm which gave the classification for surface homeomorphisms. Using this algorithm, they established a method to calculate the dilatation of a pseudo-Anosov maping class $\varphi$
. 


Dilatations themselves are related to many fields and have been intensively studied by many authors. For example, it is known that the logarithm of the dilatation of pseudo-Anosov maps is the same as the topological entropy of pseudo-Anosov maps, which is an important subject in ergodic theory. 
Also in \cite{I}, Ivanov showed that the logarithm of the asymptotic Nielsen number, which appeared in fixed point theory, coincides with the entropy. 
In this paper, we obtain a new formula to determine the dilatation of pseudo-Anosov braids from the representation $\rep$ due to B. Jiang and H. Zheng \cite{forcing_relation}. 



The \emph{growth rate} of a sequence $\{a_n\}$ of complex numbers is defined by 
\[
\Growth{n}a_n=\max\left\{1,\limsup_{n\to\infty}|a_n|^{1/n}\right\}.
\]
Let us notice that the above growth rate could be infinity. 
When the inequality $\Growth{n}a_n>1$ holds, we say that the sequence \emph{grows exponentially}. 


For any set $S$, $\mathbb{Z}S$ denotes the free abelian group with the specified basis $S$. 
If $x=\sum_{s\in S}k_ss$ is a finite sum, we define the \emph{norm} of $x$ in $\mathbb{Z}S$ by 
\[
\norm{x}=\sum_{s\in S}|k_s|. 
\]
For any matrix $A=(a_{ij})$ with coefficients in $\mathbb{Z}S$, the norm of $A$ is the matrix defined by $\norm{A}=(\norm{a_{ij}})$ when $a_{ij}$ is a finite sum for all $i$ and $j$. 


Let $\Punc$ be a finite subset of $\interior{\X}$ of $n\geq0$ points and we set $\Sn=\X\setminus\Punc$. For integers $n,m\geq0$, we consider three types of \emph{configuration spaces} as follows: The space of $m$-tuples of distinct points in $\Sn$ denoted by 
\[
\PureConf=\{(z_1,\dots,z_m)\in (\Sn)^m\mid z_i\neq z_j\:\mathrm{for\:all}\:i\neq j\}, 
\]
the space of subsets of distinct $m$ elements in $\Sn$ denoted by 
\[
\Conf=\PureConf\slash\mathcal{S}_m
\]
and the space $\IntertwiningConf$ of pairs of disjoint subsets of $n$ distinct elements and $m$ distinct elements in $\X$ denoted by 
\[
\IntertwiningConf={F_{0,n+m}(\X)}\slash\mathcal{S}_n\times\mathcal{S}_m,  
\]
where the symmetric group $\mathcal{S}_m$ acts on $\PureConf$ by permuting components of an $m$-tuple and similarly, the subgroup $\mathcal{S}_n\times\mathcal{S}_m$ of $\mathcal{S}_{n+m}$ acts on ${F_{0,n+m}(\X)}$. We write $\{y_1,\dots,y_m\}$ and $(\{x_1,\dots,x_n\},\{y_1,\dots,y_m\})$ for the elements of $\Conf$ and $\IntertwiningConf$ respectively. 

We choose $m$ distinct points $d_1,\dots,d_m$ in $\partial\X$ and take a base point $c=\{d_1,\dots,d_m\}$ of $\Conf$. Let $b=(\Punc,c)$ be a base point of $\IntertwiningConf$. 
The \emph{
$m$-braid group on $\Sn$} is defined by 
\[
\Braid=\pi_1(\Conf, c) 
\]
and the \emph{intertwining $(n,m)$-braid group on $\X$} is defined by 
\[
\IntertwiningBraid=\pi_1(\IntertwiningConf, b). 
\]


We set 
\[
\basis=\{\mu=(\mu_1,\dots,\mu_{\Euler})\in\mathbb{N}^{n-1}\mid\mu_1+\dots+\mu_{\Euler}=m\}. 
\]
We construct a $\BraidR$-invariant free $\BraidR$-submodule $\mathcal{H}_F$ of a relative homology 
of the universal covering of some configuration space generated by certain $m$-dimensional subspaces corresponding to $\mu\in\basis$. The precise definition is given in Section \ref{sec:representation}. The braid group $B_n$ acts on the homology as the mapping class group and acts on $\IntertwiningR$ by the right multiplication. Tensoring these two actions, $B_n$ acts on $\IntertwiningR\otimes_{\BraidR}\mathcal{H}_F$ and we define a representation $\rep$ by this action. 

Let $\Gamma$ be a group, $\mathbb{Z}\Gamma$ its group ring, $\Gamma_c$ the set of conjugacy classes, $\mathbb{Z}\Gamma_c$ the free Abelian group generated by $\Gamma_c$, and $\pi_\Gamma:\mathbb{Z}\Gamma\rightarrow\mathbb{Z}\Gamma_c$ the natural projection. We suppose $\zeta$ is an endomorphism of a free $\mathbb{Z}\Gamma$-module satisfying $\zeta(v_i)=\sum_{j=1}^ka_{ij}\cdot v_j$ for a basis $\{v_1,\dots,v_k\}$. The \emph{trace} of $\zeta$ is defined as 
\[
\tr_\Gamma\zeta=\pi_\Gamma\left(\sum_{i=1}^ka_{ii}\right)\in\mathbb{Z}\Gamma_c. 
\]

We note that, under the basis $\basis$, all matrix elements of $\rep(\beta)$ belong to $\mathbb{Z}\torusfund$, where $\torusfund$ is the subgroup of $B_{n+m}$ generated by $\mc$ and $\Braid$. Therefore, $\rep(\beta)$ can be naturally regarded as an endomorphism of the free $\mathbb{Z}\torusfund$-module generated by $\basis$. 

Our main result is stated as follows. 


\begin{Th}\label{Th:main_theorem} \emph{For any pseudo-Anosov braid $\mc\in B_n$, we denote by $\dil$ the dilatation of $\mc$. Then we obtain }
\[
\begin{array}{l}
\Growth{k}\norm{\tracek{k}}=\Growth{k}\tr\norm{\rep(\mc^k)}=\dil^m \\
\Growth{m}\norm{\trace}=\dil. 
\end{array}
\]
\end{Th}




The representations $\rep$ are related to homological representations of braid groups in the following way. For $m=1$, there exists a homomorphism $\Burau:\IntertwiningBraidn{1}\rightarrow\mathbb{Z}$ such that the representation induced by $\Burau$ is equivalent to the reduced Burau representation. Similarly for $m\geq2$, there exists a homomorphism $\LKB:\IntertwiningBraid\rightarrow\mathbb{Z}\oplus\mathbb{Z}$ such that the representation induced by $\LKB$ is equivalent to Lawrence-Krammer-Bigelow representation. The Lawrence-Krammer-Bigelow representations of the braid groups were studied by Lawrence \cite{Law90} in relation with Hecke algebra representations of the braid groups. In \cite{geometric_original}, \cite{algebraic_original4} and \cite{algebraic_original}, Bigelow and Krammer showed the faithfulness of the Lawrence-Krammer-Bigelow representation independently. 

In \cite{Fri86}, Fried proved that the entropy of pseudo-Anosov braids is bounded below by the logarithm of the spectral radius of the Burau matrix $B(t)$ of pseudo-Anosov braids after substituting a complex number of modulus 1 in place of $t$. 
In \cite{Kol89}, Kolev proved the same estimation directly with different methods. The estimate will be called the \emph{Burau estimate}. In \cite{Burau_estimate}, Band and Boyland showed that the spectral radius of the Burau matrix $B(t)$ of pseudo-Anosov braids after substituting the root of unity in place of $t$ is the dilatation itself of pseudo-Anosov braids only if $t=-1$. Furthermore, Band and Boyland showed that the spectral radius of $B(-1)$ is the dilatation of pseudo-Anosov braids if and only if the invariant foliations for pseudo-Anosov maps in the classes of pseudo-Anosov braids have odd order singularities at all punctures and all interior singularities are even order. 

In \cite{homological_estimate}, Koberda proved that the square of the dilatation of pseudo-Anosov braids is bounded below by the spectral radius of Lawrence-Krammer-Bigelow representation $\mathit{LKB}(q,t)$ of pseudo-Anosov braids after substituting complex numbers of modulus 1 in place of $q$ and $t$. In this paper we recover the following result of \cite{Fri86}, \cite{Kol89} and \cite{homological_estimate}. 

\begin{Th}\label{Th:recover}(Fried \cite{Fri86}, Kolev \cite{Kol89} and Koberda \cite{homological_estimate}) \emph{For a pseudo-Anosov braid $\mc$, the dilatation of $\mc$ is equal to or greater than the spectral radius of the Burau matrix $B(t)$ of $\mc$ after substituting a complex number of modulus 1 in place of $t$ and the $m$-th power of the dilatation of $\mc$ is equal to or greater than the spectral radius of the Lawrence-Krammer-Bigelow matrix $\mathit{LKB}_m(q,t)$ of $\mc$ after substituting complex numbers of modulus 1 in place of $q$ and $t$. }
\end{Th}



This paper is organized as follows. In Section \ref{sec:preliminary} we recall the definition of the topological entropy due to Adler, Konheim and McAndrew \cite{AKM65}. In Section \ref{sec:FPT}, we review asymptotic fixed point theory. We recall asymptotic fixed point theory for compact spaces due to Jiang \cite{Nielsen_theory} and a version of relative Nielsen theory due to Jiang, Zhao and Zheng \cite{stratified} and Jiang and Zheng \cite{forcing_relation}. In Section \ref{sec:trace_formula}, we construct the representation $\rep$ due to Jiang and Zheng \cite{forcing_relation} and state the relation between the trace of $\rep$ and the number of essential fixed points of some good self map. In Section \ref{sec:dilatation_proof} we prove the main theorem using the relation among dilatation, entropy and fixed point theory. In Section \ref{sec:homological_representation} we recover from our main theorem the estimation of the dilatation of pseudo-Anosov braids in \cite{Fri86}, \cite{Kol89} and \cite{homological_estimate} by means of the homological representation. 
\section{preliminaries}\label{sec:preliminary}

\subsection{Topological entropy}\label{sec:entropy}

The most widely used measure for the complexity of a dynamical system is the topological entropy. We refer the readers to \cite{Entropy} for an introductory treatment. We recall basic notions of the topological entropy due to Adler, Konheim and McAndrew \cite{AKM65}. Originally the topological entropy is defined in \cite{AKM65}. We recall \cite{AKM65} for the definition of the topological entropy. For any open cover $\alpha$ of $X$, let $N(\alpha)$ denote the number of sets in a subcover of minimal cardinality. 
For open covers $\alpha$ and $\beta$ of $X$, their join is the open cover consisting of all sets of the form $A\cap B$ with $A\in\alpha$ and $B\in\beta$. Similarly, we can define the join $\bigvee_{i=1}^n\alpha_i$ of any finite collection $\{\alpha_i\}$ of open covers of $X$. For a continuous self map $T$ of $X$, $T^{-1}\alpha$ denotes the open cover consisting of all sets $T^{-1}A$ with $A\in\alpha$. The \emph{entropy} $h(T,\alpha)$ of a map $T$ with respect to a cover $\alpha$ is defined as $\lim_{n\to\infty}\frac1n\!\:\log N\left(\bigvee_{i=0}^{n-1}T^{-i}\alpha\right)$. The \emph{topological entropy} $h(T)$ of a map $T$ is defined as $\sup h(T,\alpha)$, where the supremum is taken over all open covers $\alpha$. 

For a compact surface $X$ with negative Euler characteristic and a pseudo-Anosov homeomorphism $f$ of $X$ with the dilatation $\lambda>1$,  
$h(f)=\log\lambda$ 
is the minimal entropy in the homotopy class of $f$(\cite[p. 194]{FLP}). 

\section{Asymptotic Nielsen theory for stratified maps}\label{sec:FPT}

In \cite{Nielsen_theory}, Jiang studied fixed point theory using mapping torus. In \cite{stratified}, Jiang, Zhao and Zheng studied fixed point theory for some good noncompact spaces. In \cite{forcing_relation}, Jiang and Zheng studied fixed point theory for configuration spaces using the method in \cite{stratified}. In this section we will review some of the relevant materials from \cite{Nielsen_theory}, \cite{forcing_relation} and \cite{stratified} about fixed point theory. 

\subsection{Mapping torus}\label{sec:mapping_torus}

Subsections \ref{sec:mapping_torus} and \ref{sec:periodic} are devoted to recall basic notions of fixed point theory due to \cite{Nielsen_theory}. In \cite{Nielsen_theory}, Jiang studied fixed points by using mapping torus. 
Let $X$ be a topological space and $f:X\rightarrow X$ be a continuous self map. 
The \emph{mapping torus} $\torus{f}$ of $f$ is the space obtained from $X\times\Rplus$ by identifying $(x,s+1)$ with $(f(x),s)$ for any element $x\in X$ and $s\in\Rplus$, where $\Rplus$ stands for the real interval $[0,\infty)$. On $\torus{f}$ there exists the natural semi-flow 
\[
\varphi:\torus{f}\times\Rplus\rightarrow\torus{f},\:\varphi((x,s),t)=(x,s+t)\:\mathrm{for\:all}\:t\geq0. 
\]
A point $x$ of $X$ and a positive number $\tau>0$ determine the \emph{time-$\tau$ orbit curve} $\varphi_{(x,\tau)}=\{\varphi_t(x,0)\}_{0\leq t\leq\tau}$ in $\torus{f}$. We may identify $X$ with the cross-section $X\times\{0\}\subset\torus{f}$, then the map $f:X\rightarrow X$ is just the return map of the semi-flow $\varphi$. 

We take the base point $v$ of $X$ as the base point of $\torus{f}$. We define $\Gamma$ to be the fundamental group $\pi_1(\torus{f},v)$ of $\torus{f}$ and let $\Gamma_c$ be the set of conjugacy classes of $\Gamma$. Then $\Gamma_c$ is independent of the base point of $\torus{f}$ and can be regarded as the set of free homotopy classes of closed curves in $\torus{f}$. By the van Kampen Theorem, $\Gamma$ is obtained from $G$ by adding a new generator $z$ represented by the loop $\varphi_{(v,1)}w^{-1}$, and the relations $z^{-1}gz=f_G(g)$ for all $g\in G$:
\[
\Gamma=\langle G,z\mid gz=zf_G(g)\:\mathrm{for\:all\:}g\in G\rangle. 
\]


We note that $x$ is a fixed point of $f$ if and only if its time-$1$ orbit curve is closed on the mapping torus $\torus{f}$. For fixed points $x$ and $y$ of $f$, we define $x$ and $y$ to be in the same \emph{fixed point class} if and only if their time-$1$ orbit curves are freely homotopic in $\torus{f}$. Therefore every fixed point class $\FPC{F}$ gives rise to a conjugacy class $\cd{\FPC{F}}$ in $\Gamma_c$, called the \emph{coordinate} of $\FPC{F}$. 
A fixed point class $\FPC{F}$ is called \emph{essential} if its index $\ind{f}{\FPC{F}}$ is nonzero. 

\begin{Rem}\label{Rem:loop} \emph{We take an arbitrary path $c$ from $v$ to a fixed point $x$. In the light of the continuous map $H:I\times I\rightarrow\torus{f}$ defined by $H(s,t)=(c(t),s)$, $\varphi_{(x,1)}$ is homotopic to the loop $c^{-1}\varphi_{(v,1)}f(c)=c^{-1}zwf(c)$ and we obtain }
\[
\cd{x}=\freeloop{zwf(c)c^{-1}}, 
\]
where $\freeloop{\gamma}$ is a free homotopy class obtained by $\gamma$. 
\end{Rem}


Given a nontrivial $n$-strand braid $\mc$, there exists a connecting isotopy $\{h_t:\X\rightarrow\X\}_{0\leq t\leq1}$ from $\id$ such that the curves $\{h_t(\Punc)\}_{0\leq t\leq1}$ represent the braid $\mc$. We set $\representative=h_1$. The map $\representative$ induces a map $\Confmap{\representative}:\Conf\rightarrow\Conf$ given by 
\[
\Confmap{\representative}(\{x_1,\dots,x_m\})=\{\representative(x_1),\dots,\representative(x_m)\}. 
\]
In \cite{forcing_relation}, 
Jiang and Zheng 
showed that 
the fundamental group $\torusfund$ of $\torus{\Confmap{\representative}}$ is isomorphic to the subgroup in $B_{n+m}$ generated by $\mc$ and $\Braid$. 
\subsection{Periodic orbit classes}\label{sec:periodic}

In \cite{Nielsen_theory}, Jiang studied the periodic orbit of $f$, i.e. the fixed points of the iterates of $f$. 

The \emph{periodic point set} of $f$ is the set of points $(x,n)$ in $X\times\mathbb{N}$ satisfying $x=f^n(x)$ and is denoted by $\mathrm{PP}f$. An \emph{$n$-point} of $f$ is a fixed point $x$ of $f^n$. For an $n$-point $x$ of $f$, an \emph{$n$-orbit} of $f$ at $x$ is the $f$-orbit $\{x,\dots,f^{n-1}(x)\}$ in $X$. 
An $n$-orbit of $f$ at $x$ is a \emph{primary} $n$-orbit if $n$ is the least period of the periodic point $x$. 

An \emph{$n$-point class} of $f$ is a fixed point class $\FPC{F}^n$ of $f^n$. 
Two points $x$ and $x^\prime$ in $\Fix f^n$ are said to be in the same \emph{
$n$-orbit class
}
 of $f$ if and only if there exist natural numbers $i$ and $j$ such that $f^i(x)$ and $f^j(x^\prime)$ are in the same $n$-point class of $f$. The set $\Fix f^n$ splits into a disjoint union of $n$-orbit classes. On the mapping torus $\torus{f}$, we observe that $(x,n)$ is in the periodic point set of $f$ if and only if the time-$n$ orbit curve $\varphi_{(x,n)}$ is closed. The free homotopy class $\freeloop{\varphi_{(x,n)}}\in\Gamma_c$ of the closed curve $\varphi_{(x,n)}$ is called the \emph{$\Gamma$-coordinate} of $(x,n)$ and is denoted by $\cdGamma{x,n}$. 
Every $n$-orbit class $\FPC{O}^n$ gives rise to a conjugacy class $\cdGamma{\FPC{O}^n}$ in $\Gamma_c$, called the \emph{$\Gamma$-coordinate} of $\FPC{O}^n$. 

An important notion in the Nielsen theory for periodic orbits is the notion of reducibility. Suppose $m$ is a divisor of $n$ and $m$ is less than $n$. 
An $n$-orbit class $\FPC{O}^n$ is \emph{reducible to period $m$} if $\cdGamma{\FPC{O}^n}$ has an $(n/m)$-th root and is \emph{irreducible} if $\cdGamma{\FPC{O}^n}$ has no nontrivial root. 


An $n$-orbit class $\FPC{O}^n$ is called \emph{essential} if its index $\ind{\FPC{O}^n}{f^n}$ is nonzero. For each natural number $n$, the generalized Lefschetz number with respect to $\Gamma$ is defined as 
\[
L_\Gamma(f^n)=\sum_{\FPC{O}^n}\ind{\FPC{O}^n}{f^n}\cdot\cdGamma{\FPC{O}^n}\in\mathbb{Z}\Gamma_c, 
\]
where the summation is taken over all essential $n$-orbit classes $\FPC{O}^n$ of $f$. 
The \emph{Nielsen number of $n$-orbits} $N_\Gamma(f^n)$ is the number of nonzero terms in $L_\Gamma(f^n)$ and the indices of the essential fixed point classes appear as the coefficients in
$L_\Gamma(f^n)$. Clearly it is a lower bound for the number of $n$-orbits of $f$. The \emph{Nielsen number of irreducible $n$-orbits} $NI_\Gamma(f^n)$ is the number of nonzero primary terms in $L_\Gamma(f^n)$. It is the number of irreducible essential $n$-orbit classes. It is a lower bound for the number of primary $n$-orbits of $f$. 
These are homotopy invariants. 
\subsection{Asymptotic Nielsen theory
}\label{sec:Asymptotic}


In \cite{Nielsen_theory} Jiang defines the \emph{asymptotic Nielsen number} of $f$ to be the growth rate of the Nielsen numbers 
\[
N^\infty(f)=\Growth{n}N_\Gamma(f^n), 
\]
the \emph{asymptotic irreducible Nielsen number} of $f$ to be the growth rate of the Nielsen numbers of irreducible orbits 
\[
NI^\infty(f)=\Growth{n}N_\Gamma(f^n) 
\]
and the \emph{asymptotic absolute Lefschetz number} of $f$ to be the growth rate of the norm of generalized Lefschetz numbers 
\[
L^\infty(f)=\Growth{n}\norm{L_\Gamma(f^n)}. 
\]
In \cite{Nielsen_theory} all these asymptotic numbers are shown to enjoy the homotopy invariance. 

\begin{Rem}\label{rem:PBI}
Since the inequality $NI_\Gamma(f)\leq N_\Gamma(f)\leq\norm{L_\Gamma(f)}$ holds, we obtain $NI^\infty(f)\leq N^\infty(f)\leq L^\infty(f)$. 
In \cite{Nielsen_theory}, Jiang showed that a sufficient condition for the equality $NI^\infty(f)=N^\infty(f)$ is that $f$ satisfies the following Property of Essential Irreducibility: The number $E_n$ of essentially irreducible $n$-point classes that are reducible is uniformly bounded in $n$. Also in \cite{Nielsen_theory}, Jiang showed that a sufficient condition for the equality $N^\infty(f)
=L^\infty(f)$ is that $f$ satisfies the following Property of Bounded Index: The maximum absolute value $B_n$ of the indices of $n$-point classes $\FPC{F}_m^n$ is uniformly bounded in $n$. These conditions are not strong
. For example, every homeomorphism of $\Sn$ satisfies the Property of Essential Irreducibility and the Property of Bounded Index. 
\end{Rem}

In \cite{I}, Ivanov showed that the logarithm of the asymptotic Nielsen number $N^\infty(f)$ of a self map $f$ coincides with the entropy of a self map $f$. 

\begin{Th}\label{Th:entropy_vs_asymp}(Ivanov \cite{I}) \emph{Let $X$ be a compact surface with negative Euler characteristic and $f$ be a self map of $X$. Then the entropy of $f$ coincides with $\log N^\infty(f)$. }
\end{Th}

For a compact surface $X$ with negative Euler characteristic, we take a pseudo-Anosov homeomorphism $f$ of $X$ with the dilatation $\lambda>1$. Then 
we obtain that 
\begin{equation}\label{eq:dilatation_and_Nielsen}
h(f)=\log\lambda=\log N^\infty(f)
\end{equation}
is the minimal entropy in the homotopy class of $f$. 



\section{The representation $\rep$ and fixed points}\label{sec:trace_formula}

\subsection{The definition of $\rep$}\label{sec:representation}


In \cite{homology_of_Hm_and_IP}, Bigelow defined the triangle corresponding to the embedded edge for $m=2$. Triangles are elements of the relative homology of some abelian covering of the configuration space $\Conf$. In this subsection we define $\rep$ due to Jiang and Zheng by using the lifts of triangles to the universal covering. Let $\BraidR$ denote the group ring $\mathbb{Z}[\Braid]$ and $\IntertwiningR$ denote the group ring $\mathbb{Z}[\IntertwiningBraid]$. 


We introduce some relative homology of the universal covering of the configuration space $\Conf$. Let $\ConfCov:\Confunivcov\rightarrow\Conf$ be the universal covering of $\Conf$ and fix $\Conflift{c}\in \ConfCov^{-1}(c)$ as a base point of $\Confunivcov$. For $\varepsilon>0$, we define $V_\varepsilon$ to be the set of points $\{x_1,\dots,x_m\}$ in $\Conf$ such that at least one of the pair $(x_i,x_j)$ is within distance $\varepsilon$ of each other. We define $\Conflift{V_\varepsilon}$ to be the preimage of $V_\varepsilon$ in $\Confunivcov$. The relative homology $\Hm{\Confunivcov,\partial\Confunivcov\cup\Conflift{V_\varepsilon}}$ is nested by inclusion. 


$\Confmap{\representative}$ has a unique lift $\Conflift{\representative}:(\Confunivcov,\Conflift{c})\rightarrow(\Confunivcov,\Conflift{c})$ and induces an automorphism of the left $\BraidR$-module 
\[
\lim_{\varepsilon\to0}\Hm{\Confunivcov,\partial\Confunivcov\cup\Conflift{V_\varepsilon}}. 
\]
The induced automorphism is independent of the choice of the representative and denoted by $\Conflift{\mc}_*$. 

The intertwining $(n,m)$-braid group $\IntertwiningBraid$ is isomorphic to the subgroup $\AlgIntertwiningBraid$ of $B_{n+m}$ generated by 
\[
\sigma_1,\dots,\sigma_{n-1},\sigma_n^2,\sigma_{n+1},\dots,\sigma_{n+m-1}
\]
and $\Braid$ is isomorphic to the subgroup $\AlgBraid$ of $B_{n+m}$ generated by 
\[
A_{1,n+1},\dots,A_{n,n+1},\sigma_{n+1},\dots,\sigma_{n+m-1}, 
\]
where $A_{ij}$ is defined by 
\[
A_{ij}=\sigma_{j-1}\dots\sigma_{i+1}\sigma_i^2\sigma_{i+1}^{-1}\dots\sigma_{j-1}^{-1}. 
\]
Therefore $B_n$ acts on $\IntertwiningBraid$ by the right multiplication and so there exists an induced action of $\mc$ on the $\IntertwiningR$. Moreover, since $\Braid$ is included in $\IntertwiningBraid$, $\IntertwiningR$ is a right $\BraidR$-module. Using the $\mathbb{Z}$-module automorphism $\Conflift{\mc}_*$ and the action on $\IntertwiningBraid$ by $B_n$, we construct an automorphism $\mc\otimes\Conflift{\mc}_*$ on the left $\IntertwiningR$-module 
\[
\IntertwiningR\otimes_{\BraidR}\lim_{\varepsilon\to0}\Hm{\Confunivcov,\partial\Confunivcov\cup\Conflift{V_\varepsilon}} 
\]
by 
\[
(\mc\otimes\Conflift{\mc}_*)(h\otimes c)=h\mc\otimes\Conflift{\mc}_*(c). 
\]
%
Clearly $\mc\otimes\Conflift{\mc}_*$ is a $\IntertwiningR$-homomorphism. 


From now on, we define a representation $\rep$ of $B_n$ over the free left $\IntertwiningR$-module generated by $\basis$. 
The cardinality $d_{n,m}$ of the basis $\basis$ is $\left(
\begin{array}{c}
n+m-2 \\
m
\end{array}
\right)$. 

We now introduce some other relative homology and an intersection pairing. Henceforth every path is a continuous map from $I=[0, 1]$. For $\varepsilon>0$, we define $U_\varepsilon$ to be the set of points $\{x_1,\dots,x_m\}\in\Conf$ such that at least one of them is within distance $\varepsilon$ of some puncture point. We define $\Conflift{U}_\varepsilon$ to be the preimage of $\ConfCov$ in $\Confunivcov$. The relative homology $\Hm{\Confunivcov,\Conflift{U_\varepsilon}}$ is nested by inclusion. 


\begin{figure}[tbp]
\begin{center}
\includegraphics[width=\textwidth]{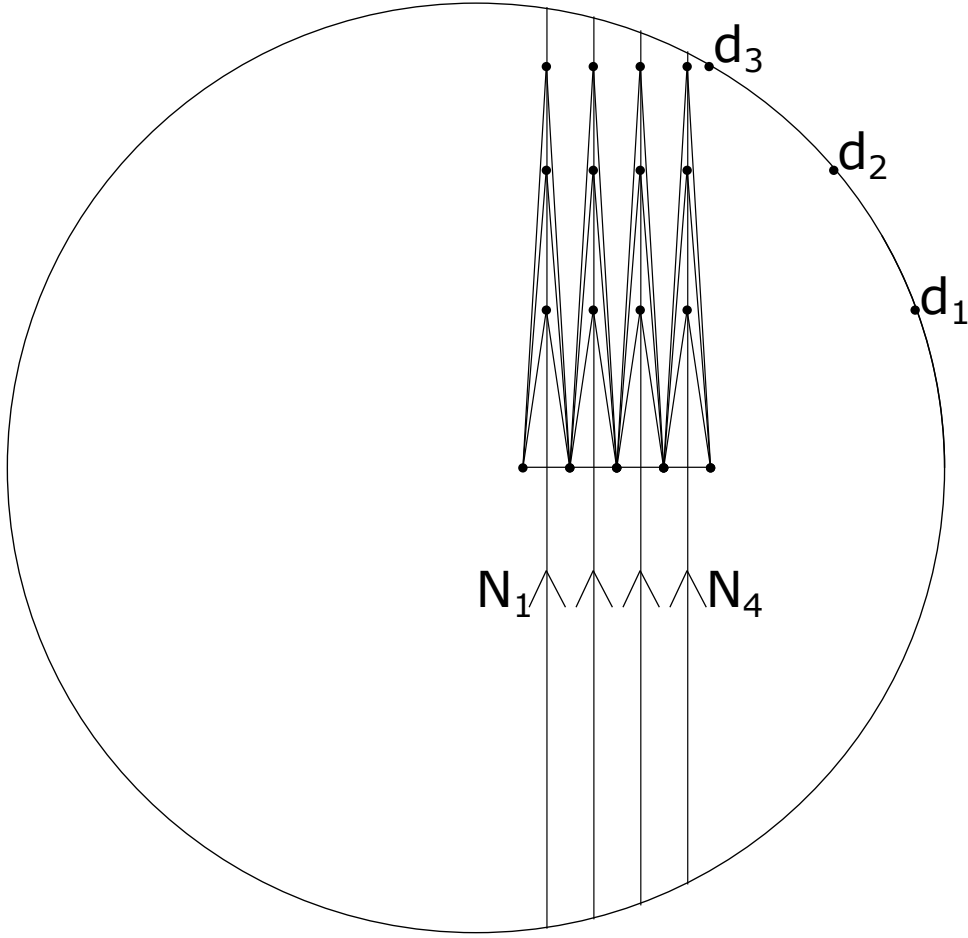}
\end{center}
\caption{The picture for $n=5$ and $m=3$}
\label{fig:forketc}
\end{figure}

We set 
\[
\begin{array}{l}
\displaystyle{p_i=\left(\frac{i}{2n},0\right), \Punc=\{p_1,\dots,p_n\},} \\
\displaystyle{d_j=\left(\cos\frac{j}{3m}\pi,\sin\frac{j}{3m}\pi\right), c=\{d_1,\dots,d_m\},} \\
\displaystyle{N_i=\left\{x=\frac{2i+1}{4n}\right\}\cap\X, \alpha_i=\left\{(x,0)\relmiddle|\frac{i}{2n}<x<\frac{i+1}{2n}\right\},} \\
\displaystyle{z_i^j=\left(\frac{2i+1}{4n},\sin\frac{j}{3m}\pi\right)} 
\end{array}
\]
and let $\alpha_i^j$ be a polygonal line connecting $p_i$, $z_i^j$ and $p_{i+1}$. 
We call $\alpha_i^j$ \emph{fork}. For $\mu\in\basis$, we set 
\[
F_\mu=\left\{\{x_1,\dots,x_m\}\in\Conf\relmiddle|\#(\{x_1,\dots,x_m\}\cap N_i)=\mu_i\right\}
\]
and 
\[
S_\mu=\prod_{i=1}^{\Euler}\prod_{j=u_i+1}^{u_{i+1}}\interior\alpha_i^j, 
\]
where $u_i=\sum_{j=1}^{i-1}\mu_j$. We take line segments $\theta_j$ on $\Sn$ from $c_j$ to $z_i^j$, where $u_i<j\leq u_{i+1}$. We notice that they are disjoint. Let $z_\mu$ be the endpoint of $\Theta_\mu=\{\theta_1,\dots,\theta_m\}$. We take a lift $\Conflift{z}_\mu$ of $z_\mu$ so that the lift $\Conflift{\Theta}_\mu$ of $\Theta_\mu$ is starting at $\Conflift{c}$ and ending at $\Conflift{z}_\mu$. We take lifts $\Conflift{F}_\mu$ and $\Conflift{S}_\mu$ of $F_\mu$ and $S_\mu$ containing $\Conflift{z}_\mu$ respectively. Let $\homology{X}$ denote the element of certain relative homology corresponding to the $m$-dimensional subspace $X$ of $\Confunivcov$. We set 
\[
\mathcal{H}_F=\bigoplus_{\mu\in\basis}\BraidR\homology{\Conflift{F}_\mu}\subset\lim_{\varepsilon\to0}\Hm{\Confunivcov,\partial\Confunivcov\cup\Conflift{V_\varepsilon}} 
\]
and 
\[
\mathcal{H}_S=\bigoplus_{\mu\in\basis}\BraidR\homology{\Conflift{S}_\mu}\subset\lim_{\varepsilon\to0}\Hm{\Confunivcov,\Conflift{U_\varepsilon}}. 
\]
For $x\in\mathcal{H}_S$ and $y\in\mathcal{H}_F$, let $(x\cdot y)\in\mathbb{Z}$ denote the standard intersection number. In  \cite{homology_of_Hm_and_IP} for $m=2$ and \cite{Jones_poly}, Bigelow defined an intersection pairing. Similarly, we define an intersection pairing 
\[
\IP{\cdot}{\cdot}:\mathcal{H}_S\times \mathcal{H}_F\rightarrow\BraidR
\mathrm{\:by\:} 
\IP{x}{y}=\sum_{\beta\in\BraidR}(x\cdot\Conflift{\beta}_*(y))\beta. 
\]
We notice that $\IP{\homology{\Conflift{S}_\mu}}{\homology{\Conflift{F}_\nu}}$ equals $1$ when $\mu=\nu$ and $0$ otherwise. Therefore $\left\{\homology{\Conflift{F}_\mu}\right\}_{\mu\in\basis}$ is linearly independent. We define elements $d_{\mu\nu}^{(\mc)}$ of $\BraidR$ so that $\{d_{\mu\nu}^{(\mc)}\}_{\mu,\nu\in\basis}$ satisfies the relations 
$\sum_\nu d_{\mu\nu}^{(\mc)}\homology{\Conflift{F}_\nu}=\Conflift{\mc}_*\left(\homology{\Conflift{F}_\mu}\right)$. 
for all $\mu\in\basis$. 
Using the intersection pairing, we obtain  
\begin{equation}\label{eq:mu,nu_entry}
d_{\mu\nu}^{(\mc)}=\tau\left(\IP{\homology{\Conflift{S}_\nu}}{\Conflift{\mc}_*\left(\homology{\Conflift{F}_\mu}\right)}\right), 
\end{equation}
where $\tau$ is an automorphism of $\BraidR$ with $\tau(\mc)=\mc^{-1}$. 
There exists a homomorphism 
\[
\rep^\prime:B_n\rightarrow\Aut_{\IntertwiningR}\left(\IntertwiningR\otimes_{\BraidR}\mathcal{H}_F\right)
\]
defined by $\rep^\prime(\mc)=(\mc\otimes\Conflift{\mc}_*)|_{\mathcal{H}_F}$. We notice that 
$\IntertwiningR\otimes_{\BraidR}\mathcal{H}_F\cong\bigoplus_{\mu\in\basis}\IntertwiningR\homology{\Conflift{F}_\mu}$
and this gives the representation $\rep$ to the matrix group 
$\GL(d_{n,m},\IntertwiningR)$. 
We set $\rep(\mc)=(c_{\mu\nu}^{(\mc)})$ and notice that $c_{\mu\nu}^{(\mc)}=\mc d_{\mu\nu}^{(\mc)}$ in $\IntertwiningR$. 
%
It is straightforward that the map $\rep$ is a group homomorphism. 


We recall the definition of trace. Let $\Gamma$ be a group, $\mathbb{Z}\Gamma$ its group ring, $\Gamma_c$ the set of conjugacy classes, $\mathbb{Z}\Gamma_c$ the free Abelian group generated by $\Gamma_c$, and $\pi_\Gamma:\mathbb{Z}\Gamma\rightarrow\mathbb{Z}\Gamma_c$ the natural projection. Let $\zeta$ be an endomorphism of a free $\mathbb{Z}\Gamma$-module satisfying $\zeta(v_i)=\sum_{j=1}^ka_{ij}\cdot v_j$ for a basis $\{v_1,\dots,v_k\}$. The \emph{trace} of $\zeta$ is defined as 
\[
\tr_\Gamma\zeta=\pi_\Gamma\left(\sum_{i=1}^ka_{ii}\right)\in\mathbb{Z}\Gamma_c. 
\]
The definition is independent of the choice of the basis
 and for two endomorphism $\zeta$ and $\xi$, we have 
$\tr_\Gamma\zeta\circ\xi=
\tr_\Gamma\xi\circ\zeta$. 

We note that, under the basis $\basis$, all matrix elements of $\rep(\beta)$ belong to $\mathbb{Z}\torusfund$. Therefore $\rep(\mc)$ can naturally be regarded as an endomorphism of the free $\mathbb{Z}\torusfund$-module generated by $\basis$. In this way, the notations $\trace$ and $\tracek{k}$ in the main theorem are well-defined. 

\begin{Th}\label{Th:main_theorem2} \emph{For any pseudo-Anosov braid $\mc\in B_n$, we denote by $\dil$ the dilatation of $\mc$. Then we obtain }
\[
\begin{array}{l}
\Growth{k}\norm{\tracek{k}}=\Growth{k}\tr\norm{\rep(\mc^k)}=\dil^m, \\
\Growth{m}\norm{\trace}=\dil. 
\end{array}
\]
\end{Th}

\subsection{The work of Jiang and Zheng}\label{sec:JZ08}


The representation $\rep$ is the same as the representation due to Jiang and Zheng \cite{forcing_relation}. We compactify $\Sn$ to a 2-disk with $n$ holes and denote it by $\Y$, and assume further that there exists a homeomorphism $\cmc:\Y\rightarrow\Y$ such that $\representative$ is the map restricting $\cmc$ on $\interior\Y$. We identify $\interior{\Y}\cup \partial\X$ with $\Sn$. We decompose the surface $\Y$ into an anulus and $n-1$ foliated rectangles, as shown in Figure \ref{fig:Yn}. 


\begin{figure}[tbp]
\begin{center}
\includegraphics[width=\textwidth]{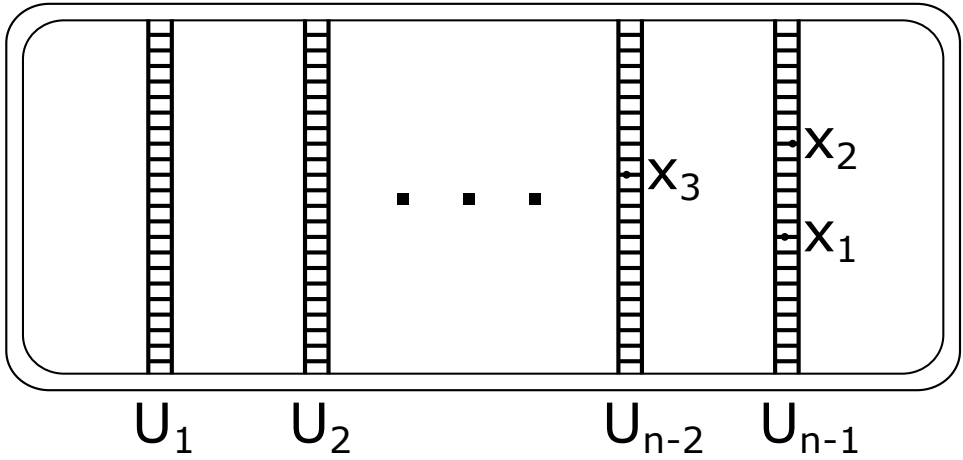}
\end{center}
\caption{Decomposition of $\Y$}
\label{fig:Yn}
\end{figure}

We define $U=U_1\cup\dots\cup U_{n-1}$ to be the union of the $n-1$ foliated open rectangles. We define a partial ordering on $U$ such that $x_1\prec x_2$ if either $x_1$ lies in a rectangle to the right of $x_2$ or $x_1$ lies in a strictly lower leaf of the same rectangle as $x_2$. For example, the order of the three points in Figure \ref{fig:Yn} is $x_1\prec x_2\prec x_3$. 

We set 
\[
V=\left\{\{x_1,\dots,x_m\}\in\C_{m,0}(\Y)\relmiddle|x_i\in U,
\begin{array}{l}
\mathrm{there\:exists\:}\eta\in\mathcal{S}_m\st \\
x_{\eta(1)}\prec\dots\prec x_{\eta(m)}
\end{array}
\right\}. 
\]
Then we have $V=\bigcup_{\mu\in\basis}V_\mu$, where 
\[
V_\mu=\left\{\{x_1,\dots,x_m\}\in V\relmiddle|
\#\{x_1,\dots,x_m\}\cap U_i=\mu_i\right\}. 
\]
Each $V_\mu$ is connected; thus the elements of $\basis$ are in one-to-one correspondence to the components of $V$. 


\begin{figure}[tbp]
\begin{center}
\includegraphics[width=\textwidth]{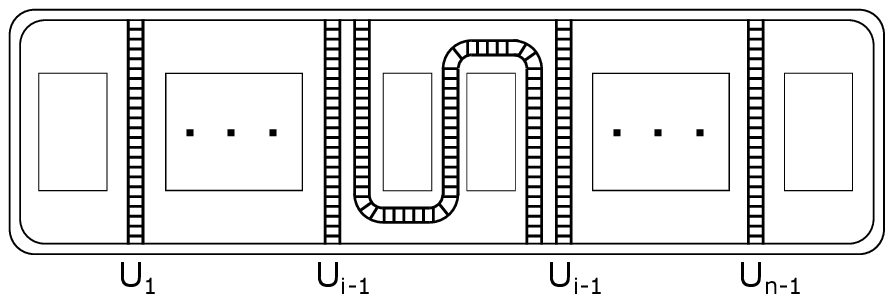}
\end{center}
\caption{The image of the self map $\phi_i$}
\label{fig:phi}
\begin{center}
\includegraphics[width=\textwidth]{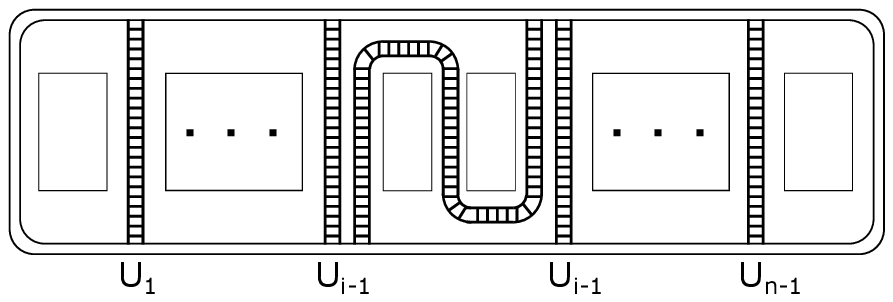}
\end{center}
\caption{The image of the self map $\overline{\phi}_i$}
\label{fig:phibar}
\end{figure}

Illustrated in Figure \ref{fig:phi} and Figure \ref{fig:phibar} are two embeddings $\phi_i$ and $\overline{\phi}_i$, which can be understood as the action of the elementary mapping $\sigma_i$ and $\sigma_i^{-1}$ on $\Y$ respectively. Both push the annulus outward, irrationally rotate the outmost boundary, keep the foliations of $(\phi_i)^{-1}(U)$ and $(\overline{\phi}_i)^{-1}(U)$, uniformly contract along the leaves of the foliations, and uniformly expand along the transversal direction. 

For every $\phi\in\{\phi_1,\dots,\phi_{\Euler},\overline{\phi}_1,\dots,\overline{\phi}_{\Euler}\}$, we have 
\[
V_\mu\cap\phi^{-1}(V_\nu)=\bigcup_{\eta\in\mathcal{S}_m}W_{\mu\nu\eta}^{(\phi)}, 
\]
where 
\[
W_{\mu\nu\eta}^{(\phi)}=\left\{x\in V_\mu\cap\phi^{-1}(V_\nu)\relmiddle|
\begin{array}{l}
\mathrm{there\:exist\:} x_1,\dots,x_m\st \\
x=\{x_1,\dots,x_m\}, \\
x_{\eta(1)}\prec\dots\prec x_{\eta(m)}, \\
\phi(x_1)\prec\dots\prec\phi(x_m), \\
\end{array}
\right\}. 
\]
Each $W_{\mu\nu\eta}^{(\phi)}$ is connected; thus the elements of the set $\{\eta\in\mathcal{S}_m\mid W_{\mu\nu\eta}^{(\phi)}\neq\emptyset\}$ are in one-to-one correspondence to the components of $V_\mu\cap\phi^{-1}(V_\nu)$. 

We choose a base point $b=\{b_1,\dots,b_m\}$ in $\interior A$. For every element $x=\{x_1,\dots,x_m\}$ in $V$ with $x_1\prec\dots\prec x_m$, the disjoint ``descending'' paths connecting $b_k$ to $x_k$ in $\Y$ give rise to a path $\gamma_x$ in $\C_{n,m}(\Y)$. Similarly, the disjoint ``ascending'' paths connecting $b_k$ to $\phi(b_k)$ give rise to a path $\gamma_{\phi(b)}$ in $\C_{n,m}(\Y)$. For every nonempty $W_{\mu\nu\eta}^{(\phi)}$, we choose a point $x\in W_{\mu\nu\eta}^{(\phi)}$ and $\alpha_{\mu\nu\eta}^{(\phi)}$ denotes the element of $\pi_1(\C_{n,m}(\Y),b)$ represented by the loop $\gamma_{\phi(b)}\cdot\phi(\gamma_x)\cdot\gamma^{-1}_{\phi(x)}$. We note that $\alpha_{\mu\nu\eta}^{(\phi)}$ is independent of the choices of $x$, $\gamma_x$, $\gamma_{\phi(b)}$ and $\gamma_{\phi(x)}$. 

In \cite{forcing_relation}, Jiang and Zheng showed that the equations 
\[
\begin{array}{l}
\displaystyle{\mu\cdot\rep(\sigma_i)=\sum_{\nu\in\basis}c_{\mu\nu}^{(i)}\cdot\nu}, \\
\displaystyle{\mu\cdot\rep(\sigma_i^{-1})=\sum_{\nu\in\basis}d_{\mu\nu}^{(i)}\cdot\nu}, 
\end{array}
\]
where 
\[
\begin{array}{l}
\displaystyle{c_{\mu\nu}^{(i)}=(-1)^{\nu_i}\cdot\sigma_i\cdot\sum_{\eta:W_{\mu\nu\eta}^{(\phi)}\neq\emptyset}\sgn\eta\cdot\alpha_{\mu\nu\eta}^{(\phi_i)}}, \\
\displaystyle{d_{\mu\nu}^{(i)}=(-1)^{\nu_i}\cdot\sigma_i^{-1}\cdot\sum_{\eta:W_{\mu\nu\eta}^{(\phi)}\neq\emptyset}\sgn\eta\cdot\alpha_{\mu\nu\eta}^{(\overline{\phi}_i)}}, 
\end{array}
\]
give rise to a group representation of $B_n$ over the free $\mathbb{Z}B_{n+m}$ module generated by $\basis$. 

We take the base point $b$ in $\Theta_\mu\cap A$. We can take the base point $b$ independent of $\mu$ because of the definition of $\Theta_\mu$ and $A$. Let $\Theta_b$ be a path from $b$ to $\Theta_\mu(1)$ along $\Theta_\mu$ and $\Theta_b^\prime$ be a path from $b$ to $\Theta_\mu(0)$ along $\Theta_\mu$. We identify $\pi_1(\Conf,c)$ with $\pi_1(\C_{0,m}(\Y),b)$ by the map induced by $\Theta_b$. 

\begin{Prop}\label{Prop:two_rep}
\emph{The representation defined above and the representation $\rep$ give the same matrix for any braid under the above identification. }
\end{Prop}
\begin{proof}
We consider the case $\mc=\sigma_i$ and the case $\mc=\sigma_i^{-1}$ is similar. We notice that $F_\mu$ is given by shrinking $V_\mu$ along the leaves of foliations and then $\Confmap{\phi}(W_{\mu\nu\eta}^{\phi})$ is homotopy equivalent to $F_\nu$. Therefore the nonzero terms of $\Conflift{\sigma_i}_*(\homology{\Conflift{F_\mu}})$ are in one-to-one correspondence to the components of $V_\mu\cap\phi^{-1}(V_\nu)$, which are in one-to-one correspondence to the elements of the set $\{\eta\in\mathcal{S}_m\mid W_{\mu\nu\eta}^{(\phi_i)}\neq\emptyset\}$. 

There exists a homotopy $\{H:\Sn\times I\rightarrow\Sn\}$ with $H(x,0)=\phi_i(x)$ and $H(x,1)=\representative(x)$ such that a map $H(\cdot,t)$ defined by $H(\cdot,t)(x)=H(x,t)$ is injective for any $t$. Let $\Confmap{H}:\Conf\times I\rightarrow\Conf$ be the map defined by $\Confmap{H}(\{x_1,\dots,x_m\},t)=\{H(x_1,t),\dots,H(x_m,t)\}$ and $\Confmap{H}(x,\cdot)$ be the path defined by $\Confmap{H}(x,\cdot)(t)=\Confmap{H}(x,t)$. 

For nonempty $W_{\mu\nu\eta}^{(\phi_i)}$, we take an element $x$ in $W_{\mu\nu\eta}^{(\phi_i)}\cap F_\mu$. We take $\gamma_x$ the composition of two paths $\Theta_b$ and the path from $z_\mu$ to $x$ in $F_\mu$. Since $\gamma_{\phi(b)}$ is homotopic to the composition of two paths $\Theta_b^\prime$ and $\Confmap{\phi_i}(\Theta_b^\prime)^{-1}$ relative to the endpoints, the loop $\Confmap{\representative}(\gamma_x)\gamma_{\Confmap{\representative}(x)}^{-1}$ is identified with $\alpha_{\mu\nu\eta}^{(\phi_i)}$ by the above identification. Therefore $\alpha_{\mu\nu\eta}^{(\phi_i)}$ is the term of $\Conflift{\sigma_i}_*(\Conflift{F_\mu})$ corresponding to $W_{\mu\nu\eta}^{(\phi_i)}$ and the signature is $(-1)^{\nu_i}\sgn\eta$. Finally, left multiplication of $\sigma_i$ and tensoring $\sigma_i$ from left induce the same action on $\IntertwiningR$. Therefore $\rep$ and the representation due to Jiang and Zheng \cite{forcing_relation} give the same matrix for all $\mc\in B_n$. 
\end{proof}

\subsection{Trace of $\rep$ and fixed points}\label{sec:trace}

In this subsection, we prove the key lemma of the proof of main theorem. We define $\eFix{f}$ to be the set of essential fixed points of $f$. We choose a word $\mc=\tau_1\dots\tau_N$, where $\tau_i$ is an element of $\{\sigma_1^{\pm1},\dots,\sigma_{n-1}^{\pm1}\}$. We put $\varphi_i=\phi_{j_i}$ if there exists a number $j_i$ satisfying $\tau_i=\sigma_{j_i}$ and $\varphi_i=\overline{\phi}_{j_i}$ if there exists a number $j_i$ satisfying $\tau_i=\sigma_{j_i}^{-1}$. Then the embedding $\emb=\varphi_N\dots\varphi_1:\Y\rightarrow\Y$ induces a map $\Confmap{\emb}:B_{n,m}(\Y)\rightarrow B_{n,m}(\Y)$. It is immediate from the definition of $\phi_i$ and $\overline{\phi}_i$ that $\Fix\Confmap{\emb}$ is a subset of $V$. 

We prove the next lemma whose proof is similar to that of \cite[Proposition 4.3.]{forcing_relation} by Jiang and Zheng. 

\begin{Lem}\label{Lem:Key_lemma}
\emph{There exists a positive number $B$ such that we have the inequality }
\[
\#\eFix(\Confmap{\emb}^k)\leq\norm{\tracek{k}}\leq B\#\eFix(\Confmap{\emb}^k). 
\]
\end{Lem}
\begin{proof}
Without loss of generality, we only have to prove the case $k=1$. We note that each of the components $W_\mu^j$ of $\bigcup_{\mu\in\basis}V_\mu\cap(\Confmap{\emb})^{-1}(V_\mu)$ is homeomorphic to $\mathbb{R}^{2m}$. Since $\Confmap{\emb}$ is a hyperbolic map on $W_\mu^j$, there exists precisely one fixed point of $\Confmap{\emb}$ on $W_\mu^j$. Let $x_j\in W_\mu^j$ be the fixed point of $\Confmap{\emb}$ on $W_\mu^j$. We notice that the fixed point class containing $x$ consists of one element $x$. We set 
\[
\alpha^g(x_j)=\gamma_{\Confmap{\emb}(c)}\cdot(\Confmap{\emb})(\gamma_{x_j})\cdot\gamma^{-1}_{x_j}. 
\]
We obtain 
\[
\cd{x_j}=\freeloop{z\gamma_{\Confmap{\emb}(c)}\cdot(\Confmap{\emb})(\gamma_{x_j})\cdot\gamma^{-1}_{x_j}}=\mc\freeloop{\alpha^g(x_j)}\in(\torusfund)_c
\]
by Remark \ref{Rem:loop} and recall that 
\[
\ind{\Confmap{\emb}}{x_j}=\langle\mathrm{diag}(\Conf),\mathrm{graph}(\Confmap{\emb})\rangle|_{x_j} 
\]
is the definition of $\ind{\Confmap{\emb}}{x_j}$. 

On the other hand, we take a lift $\Conflift{x}$ of $x$ so that the lift $\Conflift{\gamma_x}$ of $\gamma_x$ is starting at $\Conflift{c}$ and ending at $\Conflift{x}$. Then we obtain $\Conflift{\emb}(\Conflift{x_j})=\alpha^g(x_j)\Conflift{x_j}$. Computing the fixed point index $\ind{\Confmap{\emb}}{x_j}$ of $\Confmap{\emb}$ at $x_j$, we obtain 
\[
\ind{\Confmap{\emb}}{x_j}=(-1)^m\left(\alpha^g(x_j)\Conflift{S}_\mu\cdot(\Conflift{\emb})_*(\Conflift{F}_\mu)\right). 
\]
Therefore we obtain 
\[
(-1)^m\freeloop{c_{\mu\mu}^{(\mc)}}=\sum_j\ind{\Confmap{\emb}}{x_j}\cd{x_j}, 
\]
where $\freeloop{c}$ is the element of the free abelian group $\mathbb{Z}(\torusfund)_c$ projecting $c$, and 
\[
(-1)^m\trace=\sum_{x\in\Fix{\Confmap{\emb}}}\ind{\Confmap{\emb}}{x}\cdot\cd{x}. 
\]
In the above equality, the number of nonzero terms in the right hand side is $\eFix(\Confmap{g})$. By Remark \ref{rem:PBI}, there exists a positive number $B$ such that the inequality 
\[
\#\eFix(\Confmap{\emb})\leq\norm{\trace}\leq B\#\eFix(\Confmap{\emb})
\]
holds. 
\end{proof}


We count the number of essential fixed points of $\Confmap{g}^k$. Let $\{x_1,\dots,x_m\}$ be a fixed point of $\Fix{(\Confmap{\emb}^k)}$. Then there exists an $m$-tuple $(n_1,\dots,n_m)$ of natural numbers with $\sum_{i=1}^min_i=m$ such that there exist $n_i$ periodic orbits of $\emb^k$ of period $i$ in $\{x_1,\dots,x_m\}$ for all $1\leq i\leq m$. Let $A_m$ be the set of such $m$-tuples and $D_i^k$ be the number of essential periodic points of $\emb^k$ of period $i$. Then there exist $D_i^k\slash i$ periodic orbits of $\emb^k$ of period $i$ and we obtain 
\[
\#\eFix{(\Confmap{\emb}^k)}=\sum_{(n_1,\dots,n_m)\in A_m}\prod_{i=1}^m\left(
\begin{array}{c}
D_i^k\slash i \\
n_i
\end{array}
\right). 
\]
\begin{Rem}\label{rem:inequality_of_Dik}
When we consider the period of periodic points of $\emb^k$ of period $i$ as periodic points of $\emb$, we notice that $D_i^k=D_{g(k,i)i}^1$, where $g(k,i)$ is the greatest common divisor of $k$ and $i$. Moreover, if $i$ a divisor of $k$ then periodic orbits of $\emb$ of period $i$ is contained in some periodic orbits of $\emb$ of period $k$ and $D_i^1\slash i$ is equal to or greater than $D_k^1\slash k$. Therefore we have 
\[
D_i^k\slash i=D_{g(k,i)i}^1\slash i\geq D_{ki}^1\slash l(ki), 
\]
where $l(k,i)$ is the least common multiplier, and we obtain 
\[
\#\eFix{(\Confmap{\emb}^k)}\geq\sum_{(n_1,\dots,n_m)\in A_m}\prod_{i=1}^m\left(
\begin{array}{c}
D_{ki}^1\slash l(k,i) \\
n_i
\end{array}
\right). 
\]
\end{Rem}

\section{Proof of the main theorem}\label{sec:dilatation_proof}
In this section we conclude the proof of main theorem. We denote by $\dil$ the dilatation of a pseudo-Anosov braid $\mc$. 


\begin{Prop}\label{Prop:lower_bound} \emph{For any pseudo-Anosov braid $\mc\in B_n$, the inequalities }
\[
\begin{array}{l}
\Growth{k}\norm{\tracek{k}}\geq\dil^m \\
\Growth{m}\norm{\trace}\geq\dil
\end{array}
\]
\emph{hold. }
\end{Prop}
\begin{proof}
We recall that $NI_{\Gamma_{\mc^k,1}}((g^k)^i)$ defined in Section \ref{sec:periodic} is a lower bound for the number of primary $i$-orbits of $g^k$. In other words, we have the inequality $D_i^k\slash i\geq NI_{\Gamma_{\mc^k,1}}(\emb^{ki})$. When we use this inequality and Remark \ref{rem:inequality_of_Dik}, and consider the case $(n_1,\dots,n_m)=(0,\dots,0,1)$, we obtain the inequality 
\[
\begin{array}{rl}
\norm{\tr_{\torusfund}\rep(\mc^k)} & \!\!\!\!\displaystyle{\geq\#\eFix{(\Confmap{\emb}^k)}} \\
                                                 & \!\!\!\!\displaystyle{=\sum_{(n_1,\dots,n_m)\in A_m}\prod_{i=1}^m}\left(
\begin{array}{c}
D_i^k\slash i \\
n_i
\end{array}
\right) \\
                                                 & \!\!\!\!\displaystyle{\geq\frac{D_m^k}{m}=\frac{D_{km}^1}{l(k,m)}} \\
                                                 & \!\!\!\!\displaystyle{\geq g(k,m)NI_{\Gamma_{\mc,1}}(g^{km})}. 
\end{array}
\]
Since $g$ is homotopic to $\representative$, we obtain 
\[
\begin{array}{l}
\displaystyle{\Growth{k}\norm{\tr_{\torusfund}\rep(\mc^k)}\geq\Growth{k}g(k,m)NI_{\Gamma_{\mc,1}}(g^{km})=\dil^m}, \\
\displaystyle{\Growth{m}\norm{\tr_{\torusfund}\rep(\mc)}\geq\Growth{m}NI_{\Gamma_{\mc,1}}(g^m)=\dil}. 
\end{array}
\]
\end{proof}

\begin{Prop}\label{Prop:upper_bound_for_k} \emph{For any pseudo-Anosov braid $\mc\in B_n$, the inequality }
\[
\Growth{k}\tr\norm{\rep(\mc^k)}\leq\dil^m
\]
\emph{holds. }
\end{Prop}
\begin{proof}
By (\ref{eq:mu,nu_entry}), the $(\mu,\nu)$-entry of $\norm{\rep(\mc^k)}$ is $\norm{\IP{\homology{\Conflift{S}_\nu}}{\Conflift{\beta}_*^k\left(\homology{\Conflift{F}_\mu}\right)}}$. We notice that $\norm{\IP{\homology{\Conflift{S}_\nu}}{\Conflift{\mc}_*^k\left(\homology{\Conflift{F}_\mu}\right)}}$ is equal to or less than the number of intersections of $S_\nu$ and $\Confmap{\emb}^k(F_\mu)$. We define $K_{ij}^k$ to be the number of intersections of $\alpha_i^1$ and $\emb^k(N_j)$ and set $A^k=\sum_{i,j}K_{ij}^k$. We set 
\[
\rhomatrix=\left\{\rho\in M(\Euler,\mathbb{N})\relmiddle|\sum_{i=1}^{\Euler}\rho_{ij}=\nu_j, \sum_{j=1}^{\Euler}\rho_{ij}=\mu_i\right\}. 
\]
For every $i, j$ and $\rho\in\rhomatrix$, we can choose $\rho_{ij}$ paths from $\mu_i$ forks and choose one intersection from $K_{ij}$ intersections for each forks; see Figure \ref{fig:intersection}. 
Therefore we obtain 

\begin{figure}[tbp]
\begin{center}
\includegraphics[width=\textwidth]{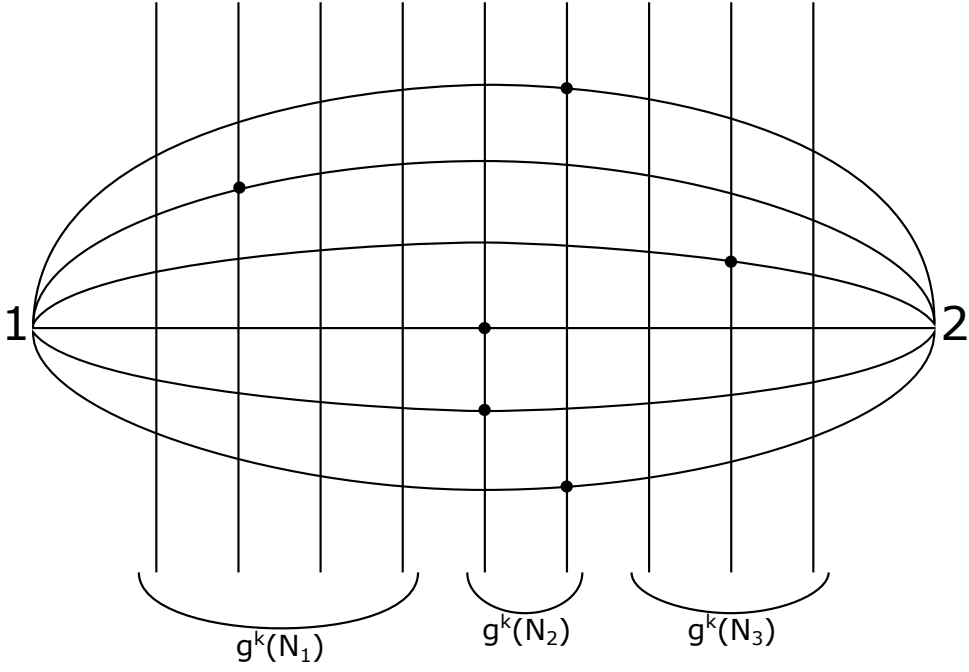}
\end{center}
\caption{The case when $\mu_1=6,K_{11}^k=4,K_{12}^k=2,K_{13}^k=3,\rho_{11}=1,\rho_{12}=4,\rho_{13}=1$}
\label{fig:intersection}
\end{figure}

\[
\begin{array}{rcl}
\norm{\IP{\homology{\Conflift{S}_\nu}}{\Conflift{\mc}_*^k\left(\homology{\Conflift{F}_\mu}\right)}} & \leq & \displaystyle{\sum_{\rho\in\rhomatrix}\prod_{i=1}^{\Euler}\mu_i!\prod_{j=1}^{\Euler}\frac{1}{\rho_{ij}!}\left(K_{ij}^k\right)^{\rho_{ij}}} \\
                                                                       & \leq & \displaystyle{\left(\prod_{i=1}^{\Euler}\mu_i\right)\sum_{\rho\in\rhomatrix}\prod_{i=1}^{\Euler}\prod_{j=1}^{\Euler}\frac{1}{\rho_{ij}!}\left(A^k\right)^{\rho_{ij}}} \\
                                                                       & = & \displaystyle{\left(\prod_{i=1}^{\Euler}\mu_i\right)\left(A^k\right)^m\sum_{\rho\in\rhomatrix}\prod_{i=1}^{\Euler}\prod_{j=1}^{\Euler}\frac{1}{\rho_{ij}!}} 
\end{array}
\]
and 
\[
\Growth{k}\norm{\IP{\homology{\Conflift{S}_\nu}}{\Conflift{\mc}_*^k\left(\homology{\Conflift{F}_\mu}\right)}}\leq\left(\Growth{k}A^k\right)^m. 
\]
It suffices to show $\Growth{k}A^k\leq\dil$. We set 
\[
U_i\cap\emb^{-1}(U_j)=\coprod_{l=1}^{K_{ij}^1}V_{ijl} 
\]
and take an open cover $\alpha=\{V_{ijk}\mid1\leq i,j\leq\Euler, 1\leq k\leq K_{ij}^1\}\cup A^\prime$ of the compact set $\Y$, where $A^\prime$ does not contain any intersections of $\emb^{-1}(\alpha_i)$ and $N_j$. 


\begin{Lem}\label{Lem:open_cover} \emph{Each element of $\bigvee_{p=0}^{k-1}\emb^{-p}(\alpha)$ contains at most one intersection of $\emb^{-k}(\alpha_j)$ and $N_i$. }
\end{Lem}
\begin{proof}
Every nonempty element of $\bigvee_{p=0}^{k-1}\emb^{-p}(\alpha)$ can be written as 
\[
B=V_{i_0i_1l_1}\cap\dots\cap\emb^{-k+1}(V_{i_{k-1} i_kl_k}) 
\]
with $i_0=i$ and $i_k=j$. By the definition of $\phi$ and $\overline{\phi}$, $\emb^k|_B:B\rightarrow U_j$ is bijective. Therefore $(\emb^k|_B)^{-1}(\alpha_j)$ is one leaf of $U_i$ and there exists only one intersection of $\emb^{-k}(\alpha_j)$ and $N_i$. 
\end{proof}
It follows from Lemma \ref{Lem:open_cover} that 
\[
A^\ell=\sum_{i,j}K_{ij}^\ell\leq N\left(\bigvee_{i=0}^{\ell-1}\emb^{-i}(\alpha)\right) 
\]
and by \eqref{eq:dilatation_and_Nielsen}, the growth rate of $N\left(\bigvee_{i=0}^{\ell-1}\emb^{-i}(\alpha)\right)$ is equal to or less than the dilatation of $\mc$. Therefore the proposition follows. 
\end{proof}

\begin{Prop}\label{Prop:upper_bound_for_m} \emph{For any pseudo-Anosov braid $\mc\in B_n$, the inequality }
\[
\Growth{m}\norm{\tr_{\torusfund}\rep(\mc)}\leq\dil
\]
\emph{holds. }
\end{Prop}
\begin{proof}
By Lemma \ref{Lem:Key_lemma}, $\norm{\tracek{k}}$ is equal to or greater than the number of essential fixed points of $\Confmap{\emb}^k$.  
For $m=1$, we notice that $\Confmap{\emb}^k$ is $g^k$. Therefore $\norm{\tr_{\Gamma_{\beta^k,1}}\zeta_{n,1}(\beta^k)}$ is equal to or greater than the number of essential periodic points of $\emb$ whose period is a divisor of $k$. In particular, we obtain $\norm{\tr_{\Gamma_{\beta^k,1}}\zeta_{n,1}(\beta^k)}\geq D_k^1\slash k$. Therefore we obtain 
\[
\begin{array}{rcl}
\norm{\tr_{\torusfund}\rep(\mc)} & \leq & \displaystyle{B\#\eFix{\Confmap{\emb}}=B\sum_{(n_1,\dots,n_m)\in A_m}\prod_{i=1}^m}\left(
\begin{array}{c}
D_i^1\slash i \\
n_i
\end{array}
\right) \\
                                             & \leq & \displaystyle{B\sum_{(n_1,\dots,n_m)\in A_m}\prod_{i=1}^m}\left(
\begin{array}{c}
\norm{\tr_{\Gamma_{\beta^i,1}}\zeta_{n,1}(\beta^i)} \\
n_i
\end{array}
\right). 
\end{array}
\]
By Proposition \ref{Prop:upper_bound_for_k}, there exists a monotonically increasing sequence $\{a_i\}$ of real numbers such that 
\[
\norm{\tr_{\Gamma_{\mc,1}}\zeta_{n,1}(\mc^i)}\leq (a_i\dil)^i\:\mathrm{and}\:\limsup_{i\to\infty}a_i=1 
\]
holds. Therefore we obtain 
\[
\begin{array}{rcl}
\norm{\tr_{\torusfund}\rep(\mc)} & \leq & \displaystyle{B\sum_{(n_1,\dots,n_m)\in A_m}\prod_{i=1}^m}(a_i\dil)^{in_i} \\
                                              & \leq & \displaystyle{B(a_m\dil)^mS_m}, 
\end{array}
\]
where $S_m$ is the number of elements of $A_m$. 
\begin{Lem}\label{Lem:summand_number} \emph{The equality $\lim_{m\to\infty}S_m^{1\slash m}=1$ holds. }
\end{Lem}
\begin{proof}
We suppose that $m-c_m>c_md_m$, where 
\[
c_m=4(\Gauss{\sqrt[4]{m}}+1)^2,\:d_m=4(\Gauss{\sqrt[4]{m}}+2) 
\]
and $\Gauss{x}$ is the floor function. Let $C_m$ be the subset of $A_m$ satisfying the following condition 
\[
\sum_{i=1}^{c_m}n_i=d_m\:\mathrm{and}\:n_{m-\sum_{i=1}^{c_m}in_i}=1. 
\]
Then $C_m$ is in one-to-one correspondence with the $d_m$-combinations with repetition from $c_m$ elements. Therefore we obtain the inequality 
\[
\begin{array}{rcl}
S_m & \geq & \left(
\begin{array}{c}
c_m+d_m-1 \\
d_m
\end{array}
\right)=\left(
\begin{array}{c}
4(\Gauss{\sqrt[4]{m}}+2)(\Gauss{\sqrt[4]{m}}+1) \\
4(\Gauss{\sqrt[4]{m}}+2)
\end{array}
\right) \\
      & = & \displaystyle{\frac{4(\Gauss{\sqrt[4]{m}}+2)(\Gauss{\sqrt[4]{m}}+1)}{4(\Gauss{\sqrt[4]{m}}+2)}\times\dots\times\frac{4(\Gauss{\sqrt[4]{m}}+1)^2}{1}} \\
      & \geq & \displaystyle{(\Gauss{\sqrt[4]{m}}+1)^{4(\Gauss{\sqrt[4]{m}}+2)}\geq\sqrt[4]{m}^{4(\Gauss{\sqrt[4]{m}}+2)}=m^{\Gauss{\sqrt[4]{m}}+2}}. 
\end{array}
\]
We set 
\[
A_{m,k}=\{(n_1,\dots,n_m)\in A_m\mid\:\max\{i\mid n_i\neq0\}=k\}. 
\]
and let $S_{m,k}$ be the number of the elements of $A_{m,k}$. Then clearly 
\[
S_m=\sum_{k=1}^mS_{m,k} 
\]
holds and the recursion formula 
\begin{equation}\label{eq:Sm}
S_{m+1,k+1}=S_{m,k}+S_{m-k,k+1} 
\end{equation}
follows from the equality $A_{m,k}=\coprod_{j=1}^kA_{m-k,j}$. Moreover, $S_{m,k}$ is less than the number of how to put $m$ balls in distinct $k$ boxes, which is $m^k$. 

We assume that $\max_kS_{m,k}=S_{m,k_0}$. Since $S_m\leq mS_{m,k_0}$ holds, we obtain 
\[
m^{k_0}\geq S_{m,k_0}\geq\frac1mS_m\geq m^{\sqrt[4]{m}} 
\]
and $k_0\geq\sqrt[4]{m}$. From (\ref{eq:Sm}), we obtain 
\[
S_{m,k_0}\leq S_{2(m-k_0),m-k_0}=S_{m-k_0}. 
\]
Since $S_m$ is monotonically increasing for $m$, we obtain 
\[
S_m\leq mS_{m,k_0}\leq mS_{m-k_0}\leq mS_{m-\sqrt[4]{m}}. 
\]
There exists a natural number $N$ such that the assumption holds for all $m\geq N$. We set $f(m)=m-\sqrt[4]{m}$ and $n_N(m)=\min\{i\mid f^i(m)\leq N\}$. Then we obtain $S_m\leq m^{n_N(m)}S_N$. We notice that if $x$ is larger than $(\sqrt[4]{m}-1)^4$, then $x-f(x)=\sqrt[4]{x}$ is larger than $\sqrt[4]{m}-1$. Therefore we obtain 
\[
f^{\Gauss{\sqrt[4]{m}^2-2\sqrt[4]{m}+2}+1}(m)\leq m-(\sqrt[4]{m}+(\sqrt[4]{m}-1)(\sqrt[4]{m}^2-2\sqrt[4]{m}+2))=(\sqrt[4]{m}-1)^4. 
\]
Therefore we obtain 
\[
n_N(m)\leq\sum_{k=1}^{\sqrt[4]{m}}\Gauss{4k^2-2k+2}+1\leq\sqrt[4]{m}(4\sqrt[4]{m}^2-2\sqrt[4]{m}+3)\leq4m^{3\slash4} 
\]
and 
\[
1<\sqrt[m]{S_m}\leq(m^{n_N(m)}S_{f^{n_N(m)}(m)})^{1\slash m}\leq\sqrt[m]{S_N}m^{4m^{-\frac14}}. 
\]
Since the limit $\lim_{m\to\infty}\sqrt[m]{S_N}m^{4\slash\sqrt[4]{m}}$ equals 1, squeeze theorem leads to the conclusion $\lim_{m\to\infty}S_m^{1\slash m}=1$. 
\end{proof}
By this lemma, we obtain 
\[
\limsup_{m\to\infty}\norm{\tr_{\torusfund}\rep(\mc)}^{1\slash m}\leq\limsup_{m\to\infty}(BS_m)^{1\slash m}a_m\dil=\dil. 
\]
\end{proof}

\noindent
{\it Proof of Theorem \ref{Th:main_theorem}. }
Since we have the inequality $\tr(\norm{A})\geq\norm{\tr A}$ for any matrix $A$ with coefficients in Laurent polynomial ring, we obtain 
\[
\dil^m\leq\Growth{k}\norm{\tracek{k}}\leq\Growth{k}\tr\norm{\rep(\mc^k)}\leq\dil^m 
\]
by Proposition \ref{Prop:lower_bound} and Proposition \ref{Prop:upper_bound_for_k}. Therefore we have 
\[
\Growth{k}\norm{\tracek{k}}=\Growth{k}\tr\norm{\rep(\mc^k)}=\dil^m. 
\]
We have 
\[
\dil\leq\Growth{m}\norm{\trace}\leq\dil 
\]
by Proposition \ref{Prop:lower_bound} and Proposition \ref{Prop:upper_bound_for_m} and we have $\Growth{m}\norm{\trace}=\dil$. $\Box$

\section{Homological representation of braid groups}\label{sec:homological_representation}

\subsection{Homological representation of braid groups}\label{sec:review}


In \cite{Law90} Lawrence construct a monodromy representation of braid groups. We review the representation. We take a homomorphism 
\[
\Burau:{\bf B}_{n,1}(\X)\cong\langle\sigma_1,\dots,\sigma_{n-1},\sigma_n^2\rangle\rightarrow\mathbb{Z} 
\]
defined by $\Burau(\sigma_i)=0$ for all $1\leq i<n$ and $\Burau(\sigma_n^2)=1$. Let $p_B:\lift{\Sn}^B\rightarrow\Sn$ be the covering corresponding to $\Ker\Burau$ and fix $\lift{d}^B\in p_B^{-1}(d_1)$. For an $n$-braid $\mc$, we take a representative $f$. Let 
\[
\lift{f}^B:(\lift{\Sn}^B,\lift{d}^B)\rightarrow(\lift{\Sn}^B,\lift{d}^B) 
\]
be the lift of $f$. Then $\lift{f}^B$ acts on $H_1(\lift{\Sn}^B,\partial\lift{\Sn}^B)$ as $\mathbb{Z}[\mathbb{Z}]$-homomorphism. The linear representation $B$ defined by $B(\mc)=\lift{f}^B_*$ is called the \emph{reduced Burau representation}. Let $t$ denote the generator of covering transformation of $\lift{\Sn}^B$ corresponding to $1\in\mathbb{Z}$. Then the ring $\mathbb{Z}[\mathbb{Z}]$ is isomorphic to the Laurent polynomial ring $\mathbb{Z}[t^{\pm1}]$ and $B(\mc)$ can be regarded as a matrix with coefficients in the Laurent polynomial ring $\mathbb{Z}[t^{\pm1}]$. Similarly for $m\geq2$, we take a homomorphism 
\[
\LKB:\Braid\cong\langle\sigma_1,\dots,\sigma_{n-1},\sigma_n^2,\sigma_{n+1},\dots,\sigma_{n+m-1}\rangle\rightarrow\mathbb{Z}\oplus\mathbb{Z} 
\]
defined by $\LKB(\sigma_i)=0\oplus0$ for all $1\leq i<n$, $\LKB(\sigma_n^2)=1\oplus0$ and $\LKB(\sigma_{n+j})=0\oplus1$ for all $1\leq j<m$. Let $p_{\mathit{LKB}}:\lift{\C}^{\mathit{LKB}}_{n,m}(\X)\rightarrow\Conf$ be the covering corresponding to $\Ker\LKB$ and fix $\lift{c}^{\mathit{LKB}}\in p_{\mathit{LKB}}^{-1}(c)$. For $\mc\in B_n$, we take a representative $f$. Let 
\[
\lift{f}^{\mathit{LKB}}:(\lift{\C}^{\mathit{LKB}}_{n,m}(\X),\lift{c}^{\mathit{LKB}})\rightarrow(\lift{\C}^{\mathit{LKB}}_{n,m}(\X),\lift{c}^{\mathit{LKB}}) 
\]
be the lift of $\Confmap{f}$. Then $\lift{f}^{\mathit{LKB}}$ acts on $H_2(\lift{B}^{\mathit{LKB}}_{n,m}(\X))$ as an $\mathbb{Z}[\mathbb{Z}\oplus\mathbb{Z}]$-homomorphism. The linear representation $\mathit{LKB}_m$ defined by $\mathit{LKB}_m(\mc)=\lift{f}^{\mathit{LKB}}_*$ is called the \emph{Lawrence-Krammer-Bigelow representations}. Let $q$ and $t$ denote the generator of covering transformation of $\lift{\C}^{\mathit{LKB}}_{n,m}(\X)$ corresponding to $1\oplus0\in\mathbb{Z}\oplus\mathbb{Z}$ and $0\oplus1\in\mathbb{Z}\oplus\mathbb{Z}$ respectively. Then the ring $\mathbb{Z}[\mathbb{Z}\oplus\mathbb{Z}]$ is isomorphic to the Laurent polynomial ring $\mathbb{Z}[q^{\pm1},t^{\pm1}]$ and $\mathit{LKB}_m(\mc)$ can be regarded as a matrix with coefficients in the 2-variable Laurent polynomial ring $\mathbb{Z}[q^{\pm1},t^{\pm1}]$. 

The homological representation of braid groups has been also intensively studied. The Lawrence-Krammer-Bigelow representations of the braid groups were studied by Lawrence \cite{Law90} in relation with Hecke algebra representations of the braid groups. In \cite{geometric_original}, \cite{algebraic_original4} and \cite{algebraic_original}, Bigelow and Krammer showed the faithfulness of the Lawrence-Krammer-Bigelow representation for $m=2$ independently. 

In \cite{Fri86}, Fried showed how to estimate the entropy of a pseudo-Anosov braid by using the Burau matrix $B(t)$ of a pseudo-Anosov braid. 
In \cite{Kol89}, Kolev proved the same estimation directly with different methods. The following theorem is the estimate and this estimate is called the \emph{Burau estimate}. 

\begin{Th}\label{Th:Burau_estimate_only}(Fried \cite{Fri86}, Kolev \cite{Kol89}) \emph{Let $f$ be a homeomorphism of $\X$ fixing $\Punc$ setwise and $\mc$ be an $n$-braid represented by $f$. Then the topological entropy of $f$ is equal to or greater than the logarithm of the spectral radius of the Burau matrix $B(t)$ of $\mc$ after substituting a complex number of modulus 1 in place of $t$. 
}
\end{Th}

If the inequality is an equality for $\eta=\eta_0$, then the Burau estimate is said to be \emph{sharp at $\eta_0$}. In \cite{Burau_estimate}, Band and Boyland determined a necessary and sufficient condition when the Burau estimate is sharp at the root of unity. 

\begin{Th}\label{Th:Burau_estimate}(Band and Boyland \cite{Burau_estimate}) \emph{For a pseudo-Anosov braid $\mc$, the Burau estimate is sharp at the root of unity $\eta_0$ only if $\eta_0=-1$. Furthermore, the Burau estimate is sharp at $-1$ if and only if the invariant foliations for a pseudo-Anosov map in the class represented by $\mc$ have odd order singularities at all punctures and all interior singularities are even order. }
\end{Th}

In \cite{homological_estimate}, Koberda shows the similar estimate by using Lawrence-Krammer-Bigelow representation. 

\begin{Th}\label{Th:LKB_estimate}(Koberda \cite{homological_estimate}) \emph{For a pseudo-Anosov braid $\mc$, the $m$-th power of the dilatation of $\mc$ is equal to or greater than the spectral radius of the Lawrence-Krammer-Bigelow matrix $\mathit{LKB}_m(q,t)$ of $\mc$ after substituting complex numbers of modulus 1 in place of $q$ and $t$. }
\end{Th}
\subsection{Homological estimation and Theorem \ref{Th:main_theorem}}\label{sec:recover}
In this section, we recover the estimation in \cite{Fri86}, \cite{Kol89} and \cite{homological_estimate} using Theorem \ref{Th:main_theorem}. If we have a homomorphism $\rho$ from $\IntertwiningBraid$ to some group $G$, we have an another representation $\rho_*(\rep)$ on the free $\mathbb{Z}[G]$-module defined by $\rho_*(\rep)=(\rho_*(c_{\mu\nu}^{(\mc)}))$. 
Moreover, if $G$ is a finitely generated free abelian group, $\mathbb{Z}[G]$ can be embedded in $\mathbb{C}$ and in this way, $\rho_*(\rep)$ gives rise to a linear representation $\rho_*^\prime(\rep)$ over $\mathbb{C}$. 

When $m=1$, Let $\Burau^\prime:\IntertwiningBraidn{1}\rightarrow\mathbb{Z}$ be a the homomorphism defined by $\Burau^\prime(\sigma_i)=0$ for all $1\leq i<n$ and $\Burau^\prime(\sigma_i^2)=1$. When $m\geq2$, let $\LKB^\prime:\IntertwiningBraid\rightarrow\mathbb{Z}\oplus\mathbb{Z}$ be a homomorphism defined by $\LKB^\prime(\sigma_i)=0\oplus0$ for all $1\leq i<n$, $\LKB^\prime(\sigma_n^2)=1\oplus0$ and $\LKB^\prime(\sigma_{n+j})=0\oplus1$. 
We consider the homomorphism from $\Aut_{\IntertwiningR}\left(\IntertwiningR\otimes_{\BraidR}\mathcal{H}_F\right)$ induced by $\LKB^\prime$. Since $\LKB^\prime(\sigma_i)$ is $0\oplus0$ for all $1\leq i<n$, the action as the right multiplication becomes trivial and $(\LKB^\prime)_*(\rep)$ is equivalent to the Lawrence-Krammer-Bigelow representations for all $m\geq2$. Similarly, $(\Burau^\prime)_*(\rep)$ is equivalent to the reduced Burau representation. 

For any matrix $A$ with coefficients in $n$-variable Laurent polynomial ring and complex numbers $x_1,\dots,x_n$, we denote by $A(x_1,\dots,x_n)$ the matrix with coefficients in $\mathbb{C}$ substituting $x_i$ for $i$-th variable. For any matrix $A$ with coefficients in $\mathbb{C}$, we denote by $\sr A$ the spectral radius of $A$. We state the main result of this section. 

\begin{Prop}\label{Prop:trace_estimate} \emph{For any matrix $A$ with coefficients in the Laurent polynomial ring $\mathbb{Z}[x_1,\dots,x_n]$, we have }
\[
\Growth{k}\norm{\tr A^k}=\sup_{x_i\in S^1}\sr A(x_1,\dots,x_n). 
\]
\end{Prop}
Let $I=(i_1,\dots,i_n)$ be a multi index and $x^I=\prod_{k=0}^nx_k^{i_k}$. 

\begin{Lem}\label{Lem:polynomial} \emph{We suppose $f(x_1,\dots,x_n)=\sum_{i_1=0}^M\dots\sum_{i_n=0}^Ma_Ix^I$ is an $n$-variable polynomial of degree $M$. Then we have the inequality }
\[
\sum_I|a_I|\leq(M+1)^n\sup_{x_k\in S^1}|f(x_1,\dots,x_n)|
\]
\end{Lem}
\begin{proof}
First of all, we prove the case $n=1$. Then $f(x)$ is a polynomial $\sum_{i=0}^Ma_ix^i$ of degree $M$. We consider the Vandermonde matrix 
\[
V=V_{M+1}(x_0,\dots,x_M)=\left(
\begin{array}{cccc}
1 & x_0 & \cdots & x_0^M \\
1 & x_1 & \cdots & x_1^M \\
\vdots & \vdots & \ddots & \vdots \\
1 & x_M & \cdots & x_M^M \\
\end{array}
\right). 
\]
Then we have $V{\bf a}={\bf A}$, where 
\[
{\bf a}=\left(
\begin{array}{c}
a_0 \\
\vdots \\
a_M
\end{array}
\right)\mathrm{\:and\:}{\bf A}=\left(
\begin{array}{c}
f(x_0) \\
\vdots \\
f(x_M)
\end{array}
\right). 
\]
We denote by $\sigma_m$ the $m$-th elementary symmetric function in the $(M+1)$ variables $x_0,\dots,x_M$. In other words, we have 
\[
\sigma_m=\sigma_m(x_0,\dots,x_M)=\sum_{\nu\in\mathcal{S}_m} x_{\nu(1)}\dots x_{\nu(m)}
\]
for all $1\leq m\leq M+1$ and $\sigma_0=1$. We use the notation $\sigma_m^i$ to denote the $m$-th elementary symmetric function in the $M$ variables $x_k$ with $x_i$ missing. In other words, we have 
\[
\sigma_m^i=\sigma_m(x_0,\dots,x_{i-1},x_{i+1},\dots,x_M). 
\]
We set $V^{-1}=(v_{ij})_{0\leq i,j\leq M}$. It is well known (see \cite{Vandermonde}) that we have 
\[
v_{ij}=(-1)^i\frac{\sigma_{M-i}^j}{\prod_{k\neq j}(x_k-x_j)} 
\]
We put $\theta=\pi\slash M+1$ and $x_k=\exp(2\sqrt{-1}k\theta)$. 
Since $x_i$'s are all the roots of $z^{M+1}-1=0$, we obtain $\sigma_m(x_0,\dots,x_M)=0$ for all $1\leq m\leq M$. Since the recursion formula $\sigma_{m+1}^i=\sigma_{m+1}-x_i\sigma_m^i$ holds, we obtain $\sigma_{m+1}^i=-x_i\sigma_m^i$ and $\sigma_m^i=(-x_i)^m$. 
Then we obtain 
\[
|v_{ij}|=\left|(-1)^{i-1}\frac{\sigma_{M-i}^j}{\prod_{k\neq j}(x_k-x_j)}\right|=\frac{1}{\prod_{k=1}^M(2\sin k\theta)}. 
\]
Since we have ${\bf a}=V^{-1}{\bf A}$, we have the inequality 
\[
|a_i|\leq\frac{M+1}{\prod_{k=1}^M(2\sin k\theta)}\max_k\left|f(x_k)\right|\leq\frac{M+1}{\prod_{k=1}^M(2\sin k\theta)}\sup_{x\in S^1}\left|f(x)\right|. 
\]
\begin{Lem}\label{Lem:sine} \emph{The equality $\prod_{k=1}^M(2\sin k\theta)=M+1$ holds. }
\end{Lem}
\begin{proof}
We set 
\[
\cos(2n-1)\theta=\cos\theta f_n(\cos\theta),\:\sin2n\theta=\sin2\theta g_n(\cos\theta)
\]
for $n\geq1$. Since 
\[
\left\{
\begin{array}{l}
\cos(2n+3)\theta+\cos(2n-1)\theta=2\cos2\theta\cos(2n+1)\theta \\
\sin2(n+2)\theta+\sin2n\theta=2\cos2\theta\sin2(n+1)\theta, 
\end{array}\right.
\]
hold, we obtain recursion formulae $f_{n+2}(x)=2(2x^2-1)f_{n+1}(x)-f_n(x)$ and $g_{n+2}(x)=2(2x^2-1)g_{n+1}(x)-g_n(x)$. Moreover, because of the initial conditions $f_1(x)=1$, $f_2(x)=4x^2-3$, $g_1(x)=1$ and $g_2(x)=4x^2-2$, $f_n(x)$ and $g_n(x)$ are polynomials of degree $2(n-1)$. Solving the recursion formulae of leading coefficient and constant term, we find that the leading coefficients of $f_n(x)$ and $g_n(x)$ is $4^n$, the constant term of $f_n(x)$ is $(2n-1)(-1)^{n-1}$ and the constant term of $g_n(x)$ is $n(-1)^{n-1}$. 

There exist distinct $2(n-1)$ solutions 
\[
\pm\sin(k\pi\slash(2n-1))=\cos(\pi/2\pm k\pi\slash(2n-1))\:\:k=1,\dots,n-1
\]
of $f_n(x)=0$ and distinct $2(n-1)$ solutions 
\[
\pm\sin(k\pi\slash2n)=\cos(\pi/2\pm k\pi\slash2n)\:\:k=1,\dots,n-1
\]
of $g_n(x)=0$. Vieta's formula implies $\prod_{k=1}^M(2\sin k\theta)=M+1$. 
\end{proof}
Lemma \ref{Lem:sine} implies $\sum_{i=0}^M|a_i|\leq(M+1)\sup_{x\in S^1}|f(x)|$. 

Now we consider the general case. For any $n$-variable polynomial 
\[
f(x_1,\dots,x_n)=\sum_{i_1=0}^M\dots\sum_{i_n=0}^Ma_Ix^I 
\]
of degree $M$, we set 
\[
f(x_1,\dots,x_n)=\sum_{i_n=0}^Mf_{i_n}(x_1,\dots,x_{n-1})x_n^{i_n}. 
\]
Then we obtain 
\[
\sup_{x_1,\dots,x_{n-1}\in S^1}\sum_i|f_i(x_1,\dots,x_{n-1})|\leq(M+1)\sup_{x_1,\dots,x_n\in S^1}|f(x_1,\dots,x_n)|. 
\]
Repeating this $n$ times shows the inequality 
\[
\sum_I|a_I|\leq(M+1)^n\sup_{x_1,\dots,x_n\in S^1}|f(x_1,\dots,x_n)|. 
\]
\end{proof}

\noindent
{\it Proof of Proposition \ref{Prop:trace_estimate}. }
We notice that 
\[
\sup_{x_i\in S^1}\left|\sum_{i_1=m}^M\dots\sum_{i_n=m}^Ma_Ix^I\right|=\sup_{x_i\in S^1}\left|\sum_{i_1=0}^{M-m}\dots\sum_{i_n=0}^{M-m}a_Ix^I\right| 
\]
holds. We denote by $A$ a matrix 
with coefficients in $n$-variable Laurent polynomial ring. Let $M$ and $m$ be the maximum and minimum degree of all entries of $A$. Then the maximum degree of all entries of $A^k$ is equal to or less than $kM$ and the minimum degree of all entries of $A^k$ is equal to or greater than $km$. Using Lemma \ref{Lem:polynomial}, we obtain 
\[
\sup_{x_i\in S^1}|\tr A^k(x_1,\dots,x_n)|\leq\norm{\tr A^k}\leq(k(M-m)+1)^n\sup_{x_i\in S^1}|\tr A^k(x_1,\dots,x_n)|. 
\]
Therefore we obtain 
\[
\Growth{k}\norm{\tr A^k}=\Growth{k}\sup_{x_i\in S^1}|\tr A^k(x_1,\dots,x_n)|. 
\]
Cayley-Hamilton theorem shows
\[
\tr A^k(x_1,\dots,x_n)=\lambda_1^k+\dots+\lambda_N^k, 
\]
where $\lambda_1,\dots,\lambda_N$ are the eigenvalues of $A(x_1,\dots,x_n)$. Therefore we obtain 
\[
\Growth{k}\sup_{x_i\in S^1}|\tr A^k(x_1,\dots,x_n)|=\sup_{x_i\in S^1}\sr A(x_1,\dots,x_n). 
\]
$\Box$

Using Proposition \ref{Prop:trace_estimate}, we recover the estimation in \cite{Fri86}, \cite{Kol89} and \cite{homological_estimate}. 

\begin{Cor}\label{Cor:recover} \emph{For a pseudo-Anosov braid $\mc$, the dilatation of $\mc$ is equal to or greater than the spectral radius of the Burau matrix $B(t)$ of $\mc$ after substituting a complex number of modulus 1 in place of $t$ and the $m$-th power of the dilatation of $\mc$ is equal to or greater than the spectral radius of the Lawrence-Krammer-Bigelow matrix $\mathit{LKB}_m(q,t)$ of $\mc$ after substituting complex numbers of modulus 1 in place of $q$ and $t$. }
\end{Cor}
\begin{proof}
Since $\norm{\tr(\rho)_*(\rep)(\mc^k)}$ is equal to or less than $\norm{\tracek{k}}$, we obtain 
\[
\Growth{k}\norm{\tr(\Burau^\prime)_*(\zeta_{n,1})(\mc^k)}\leq\dil
\]
and 
\[
\Growth{k}\norm{\tr(\LKB^\prime)_*(\rep)(\mc^k)}\leq\dil^m. 
\]
From Proposition \ref{Prop:trace_estimate}, we obtain 
\[
\Growth{k}\norm{\tr(\Burau^\prime)_*(\zeta_{n,1})(\mc^k)}=\sup_{t\in S^1}B(t) 
\]
and 
\[
\Growth{k}\norm{\tr(\LKB^\prime)_*(\rep)(\mc^k)}=\sup_{q,t\in S^1}\mathit{LKB}_m(q,t). 
\]
Therefore we obtain 
\[
\sup_{t\in S^1}B(t)\leq\dil\mathrm{\:\:\:and\:}\sup_{q,t\in S^1}\mathit{LKB}_m(q,t)\leq\dil^m. 
\]
\end{proof}

On the other hand, it is not known whether $\Growth{m}\norm{\tr(\LKB)_*(\rep)(\beta)}$ is $\dil$ or not. If $\Growth{m}\norm{\tr(\LKB)_*(\rep)(\beta)}$ is not necessarily $\dil$, there exists some sufficient condition for $\Growth{m}\norm{\tr(\LKB)_*(\rep)(\beta)}=\dil$. Clearly the condition in Theorem \ref{Th:Burau_estimate} is a sufficient condition for the above equality. We want to reveal whether this sufficient condition is the best condition or not. 

\section*{Acknowledgement}

I would like to show my greatest appreciation to Professor Toshitake Kohno whose comments and suggestions were of inestimable value for my study. Special thanks also go to the member of the same seninar whose opinions and information have helped me very much throughout the production of this study. I would also like to express my gratitude to my family for their moral support and warm encouragements.


\end{document}